\documentclass[11pt, twoside, leqno]{amsart}  
\usepackage{lipsum}
\usepackage{amsfonts}
\usepackage{graphicx}
\usepackage{epstopdf}
\usepackage{algorithmic}
\usepackage{calligra}
\usepackage[dvipsnames]{xcolor}
\usepackage{amsfonts,amsmath,amsthm,amssymb}
\usepackage{mathtools}
\usepackage{hyperref}
\usepackage[makeroom]{cancel}
\usepackage{autonum}
\usepackage{hhline}
\usepackage{array}
\usepackage{diagbox}
\usepackage{mdframed}
\usepackage{multicol}
\usepackage{graphicx}
\usepackage{subcaption}
\usepackage{moreverb}
\usepackage{bbm}
\usepackage[margin=1.38in]{geometry}
\usepackage{todonotes}
\usepackage{scalerel,amssymb}
\allowdisplaybreaks
\usepackage{mathrsfs}  
\usepackage{lineno}
\usepackage{todonotes}
\usepackage{tikz}
\usepackage{appendix}
\usepackage{enumitem}
\usepackage{pgfplots}
\usetikzlibrary{arrows.meta}
\usepackage[numbers,sort&compress]{natbib}
\definecolor{mygreen}{HTML}{43a047}
\usepackage{subcaption}
\usepackage{doi}
\newcommand{\Om}{\Omega}

\newcommand{\rhob}{\rho_{\textup{b}}}
\newcommand{\rhoa}{\rho_{\textup{a}}}
\newcommand{\Ca}{C_{\textup{a}}}
\newcommand{\Cb}{C_{\textup{b}}}
\newcommand{\kappaa}{\kappa_{\textup{a}}}
\newcommand{\Thetaa}{\Theta_{\textup{a}}}

\newcommand{\ddt}{\frac{\textup{d}}{\textup{d}t}}

\newcommand{\ds}{\, \textup{d} s }
\newcommand{\dx}{\, \textup{d} x}

\newcommand{\intO}{\int_{\Omega}}
\newcommand{\R}{\mathbb{R}} 
 
\newtheorem{theorem}{Theorem}
\newtheorem{lemma}{Lemma}
\newtheorem{proposition}{Proposition}
\newtheorem{assumption}{Assumption}

\numberwithin{lemma}{section}
\numberwithin{proposition}{section}
\numberwithin{theorem}{section}
\numberwithin{equation}{section}
\makeatletter
\newcommand{\leqnomode}{\tagsleft@true}
\newcommand{\reqnomode}{\tagsleft@false}
\makeatother
\definecolor{grey}{rgb}{0.5,0.5,0.5}
   
\title[JMGT--Pennes system]{Local well-posedness of a coupled Jordan--Moore--Gibson--Thompson--Pennes  model of nonlinear ultrasonic heating}      
     
            
\author[I. Benabbas]{Imen Benabbas$^\dagger$}
\thanks{$^\dagger$AMNEDP Laboratory, Faculty of Mathematics,
	USTHB (\href{ibenabbas@usthb.dz}{ibenabbas@usthb.dz})}
\author[B. Said-Houari]{Belkacem Said-Houari$^\ddag$}
\thanks{$^\ddag$Department of Mathematics, College of Sciences, University of
	Sharjah, P. O. Box: 27272, Sharjah, United Arab Emirates    (\href{bhouari@sharjah.ac.ae}{bhouari@sharjah.ac.ae})}
\begin{document}
	\vspace*{8mm}
	\begin{abstract}
In this work, we investigate a mathematical model of nonlinear ultrasonic heating based on the Jordan--Moore--Gibson--Thompson equation (JMGT) with temperature-dependent medium parameters coupled to the semilinear Pennes equation for the bioheat transfer. The equations are coupled via the temperature in the coefficients of the JMGT equation and via a nonlinear source term within the Pennes equation, which models the absorption of acoustic energy by the 
 surrounding tissue.  
Using the energy method together with a fixed point argument, we prove that our model is locally well-posed, provided that the initial data are regular, small in a lower topology and the final time is short enough.      
					\end{abstract}   
	\vspace*{-7mm}  
	\maketitle                
     
\section{Introduction}    
The study of nonlinear wave propagation has garnered a lot of attention in recent years, notably due to recent applications in ultrasound imaging such as High-Intensity Focused Ultrasound (HIFU). 
%%%%%%
 For instance, in medical procedures, focused ultrasound is used to generate localized heating that can destroy the targeted region. Indeed, this technique is proving its success in the treatment of both benign and malignant tumors \cite{ hahn2018high, li2010noninvasive, maloney2015emerging}. This emerging method relies heavily on the ability to model accurately the nonlinear propagation of sound pulses in thermo-viscous elastic media.  
%%%%%%%%%%
 
In this paper, we consider the coupled system of  the Jordan--Moore--Gibson--Thompson equation  and the  Pennes bioheat equation (JMGT--Pennes  model):
\begin{subequations}\label{coupled_problem_JMGT_Pennes} 
\begin{equation} 
\left\{ \label{coupled_problem_JMGT_Pennes_1}
\begin{aligned}
&\tau p_{ttt}+p_{tt}-c^2(\bar{\Theta})\Delta p - \delta(\bar{\Theta}) \Delta p_t = K(\bar{\Theta})\left(p^2\right)_{tt}, \quad &\text{in} \ \Omega \times (0,T),\\
&\rhoa \Ca\bar{\Theta}_t-\kappaa \Delta \bar{\Theta} + \rhob \Cb W(\bar{\Theta}-\Thetaa) = \mathcal{Q}(\bar{\Theta}, p_t), \quad &\text{in} \ \Omega \times (0,T).
\end{aligned}
\right.
\end{equation}
The first equation in \eqref{coupled_problem_JMGT_Pennes_1} is the Jordan--Moore--Gibson--Thompson equation of nonlinear acoustics with  temperature-dependent medium parameters, and the second equation is the Pennes bioheat equation \cite{pennes1948analysis}. The two equations are coupled via the temperature in the sound coefficients of the JMGT equation and via the function $\mathcal{Q}$ in the  Pennes equation modeling the acoustic energy being absorbed by the surrounding tissue. 
 Here $p$ and $\bar{\Theta}$ denote the acoustic pressure and the temperature fluctuations, respectively,  $c(\bar{\Theta})$ is the speed of sound, which is allowed to depend on the temperature,  and 
 \begin{equation} \label{delta_definition}   
\delta(\bar{\Theta})=\beta(\bar{\Theta})+\tau c^2(\bar{\Theta})
\end{equation}
where $\beta$ is a strictly positive function which represents the sound diffusivity and it is given by 
\begin{equation} \label{beta_definition}
\beta(\bar{\Theta})= 2\frac{\alpha c^3(\bar{\Theta})}{\omega^2},
\end{equation}
where $\omega$ stands for the angular frequency and $\alpha$ for  the acoustic amplitude absorption coefficient \cite{connor2002bio}.
The parameter   
 $\tau>0$ denotes  a positive constant accounting for the relaxation time, and   the function $K(\bar{\Theta})$  is also allowed to depend on $\bar{\Theta}$ and it is given by
$$K(\bar{\Theta})=\frac{\beta_{\text{acous}}}{\rho c^2(\bar{\Theta })},$$
where $\rho$ is the mass density and $\beta_{\text{acous}}$ is the parameter of nonlinearity. 
 The source term in the second equation $\mathcal{Q}(\bar{\Theta}, p_t)$ represents the acoustic energy absorbed by the tissue. 
We assume that  the function $\mathcal{Q}$ has the form \cite{connor2002bio}
\begin{equation} \label{Q_definition}
\mathcal{Q}(\bar{\Theta},p_t)=q(\bar{\Theta}) (p_t)^2 \quad \text{with} \quad  q(\bar{\Theta})=\frac{2 \beta(\bar{\Theta})}{\rhoa c^4(\bar{\Theta})}. 
\end{equation}
The medium parameters $\rhoa, \Ca$ and $\kappaa$ stand, respectively, for the ambient density, the ambient heat capacity and thermal conductivity of the tissue. The additional term $\rhob \Cb W(\bar{\Theta}-\Thetaa)$ accounts for the heat loss due to blood circulation, with  $\rhob, \Cb$ being the density and specific heat capacity of blood, and $W$ expressing the tissue's volumetric perfusion rate measured in milliliters of blood per milliliter of tissue per second.

We supplement    \eqref{coupled_problem_JMGT_Pennes_1} with the initial conditions 
\begin{eqnarray}\label{Initial_Condi}
p|_{t=0}=p_0,\quad p_t|_{t=0}=p_1,\quad p_{tt}|_{t=0}=p_2,\quad \bar{\Theta}|_{t=0}=\bar{\Theta}_0.
\end{eqnarray}
and Dirichlet boundary conditions  
\begin{eqnarray} \label{coupled_problem_BC}
p\vert_{\partial \Om}=0, \qquad \bar{\Theta}\vert_{\partial \Om}= \Thetaa, 
\end{eqnarray}
\end{subequations} 
with $\Thetaa$ denoting the ambient temperature, that is typically taken in the human body to be $37^\circ C$; see \cite{connor2002bio} for more details.   

When $c$, $\delta$ and $K$ are constants, 
the first equation in   \eqref{coupled_problem_JMGT_Pennes_1} reduces to  the Jordan--Moore--Gibson--Thompson (JMGT) \cite{Jordan_2014},  with the Westervelt nonlinearity written in terms of the acoustic pressure: 
%\label{Main_problem}
\begin{equation}
\tau p_{ttt}+p_{tt}-c^{2}\Delta p-\delta \Delta p_{t}=K\left(p^2\right)_{tt} ,  \label{MGT_1}
\end{equation}%
see \cite{Jordan_2014, Kaltenbacher2015MathematicsON} and reference therein for the derivation of \eqref{MGT_1}. 

Equation \eqref{MGT_1} is an alternative model to the classical  Westervelt equation \cite{westervelt1963parametric}
\begin{equation}\label{Wester}
p_{tt}-c^{2}\Delta p-\delta \Delta p_{t}=K\left(p^2\right)_{tt}. 
\end{equation}
By taking into account the nonlinear effect modeled by a quadratic velocity term,  we obtain the JMGT equation with Kuznetsov nonlinearity (written in terms of the acoustic velocity potential $\psi$):  
\label{Main_problem}
\begin{equation}
\tau \psi_{ttt}+\psi_{tt}-c^{2}\Delta \psi-\delta \Delta \psi_{t}=\frac{\partial }{%
\partial t}\left( K(\psi_{t})^{2}+|\nabla
\psi|^{2}\right) .  \label{MGT_1_1}
\end{equation}%
The presence of the term $\tau p_{ttt}$ in \eqref{MGT_1} (resp. $\tau\psi_{ttt}$ in \eqref{MGT_1_1}) arises from employing the constitutive Cattaneo law of heat conduction instead of the classical Fourier law during the derivation process of the Westervelt (resp. Kuznetsov) equation. This, of course, changes  completely the nature of the problem,  since in equation \eqref{Wester}, the presence of the term $-\delta \Delta p_{t}$ yields a parabolic structure of the equation, thereby inducing a regulazation effect. This is not the case for the hyperbolic type equations \eqref{MGT_1} and \eqref{MGT_1_1}, hence requiring different types of estimates than the related second-order equations. 

%%%%
In recent years, there has been a lot of work investigating the JMGT equation and its linearized version; the Moore--Gibson--Thompson (MGT) equation.
%%%
Concerning some studies for the  MGT or JMGT equations, we refer the interested  readers to  \cite{bucci2019feedback, bucci2019regularity, conejero2015chaotic,Kal_Las_Mar,Lizama_Zamorano_2019, Trigg_et_al, PellSaid_2019_1, P-SM-2019,kaltenbacher2012well}, and the references therein.  
The MGT and JMGT equations with a memory term have been also investigated recently. For the  MGT with memory, the reader is referred to  \cite{Bounadja_Said_2019,Liuetal._2019,dell2016moore} and to \cite{lasiecka2017global,nikolic2020mathematical,Nikolic_SaidHouari_2} for the JMGT with memory.  
%%%%%%%
The singular limit problem when $\tau\rightarrow 0$ has been rigorously justified in \cite{KaltenbacherNikolic} where the authors   showed that   the limit of \eqref{MGT_1_1}  as $\tau \rightarrow 0$ leads to the Kuznetsov equation. See also  the recent contributions   \cite{Bongarti_All,bongarti2021vanishing}. The study of the controllability properties of the MGT type equations can be found for instance in \cite{bucci2019feedback, Lizama_Zamorano_2019}.

The global well posedness and large time behaviour of the solution to the Cauchy problem associated to the nonlinear 3D model  \eqref{MGT_1_1} has been recently investigated in \cite{CHEN2023113316, Racke_Said_2019,Said_Houari_2022_1,Said-Houari:2022aa} where  small data global existence results in   various
function spaces have been established. We also point out the recent result in \cite{Nikolic_Kaltenbacher_Fractional}, where a fractional models associated to \eqref{MGT_1_1} were    derived and analyzed.

For thermal modeling, the first equation in \eqref{coupled_problem_JMGT_Pennes_1} is coupled with the Pennes bioheat equation:
\begin{equation} \label{Heat_Eq}
	\rhoa \Ca\bar{\Theta}_t -\kappaa\Delta \bar{\Theta}+ \rhob \Cb W(\bar{\Theta}-\Thetaa) = 0.
\end{equation}
 Equation \eqref{Heat_Eq} is a widely used model for studying heat transfer in biological systems.  It  takes into account the heat transfer by conduction in the tissues and the convective heat transfer due to blood perfusion.  See ~\cite{pennes1948analysis} for the derivation of \eqref{Heat_Eq}.

 The presence of the constant $\tau$ in \eqref{coupled_problem_JMGT_Pennes} is due to the use of the Cattaneo model for hyperbolic heat conduction   in the derivation of the JMGT equation  
 from the equations of fluid dynamics. It is well known that the Cattaneo law  eliminates the so-called infinite speed of propagation paradox observed when parabolic heat models are used \cite{Ca48,Ch98}.
  In fact, when $\tau=0$, system \eqref{coupled_problem_JMGT_Pennes_1} reduces to the  Westervelt--Pennes--Fourier system: 
 \begin{equation}\label{Westervelt--Pennes--Fourier}
  \left\{
  \begin{aligned}
	&p_{tt}-c^2(\bar{\Theta})\Delta p - b \Delta p_t = k(\bar{\Theta}) \left(p^2\right)_{tt},\\
	&\rhoa \Ca\bar{\Theta}_t -\kappaa\Delta \bar{\Theta}+ \rhob \Cb W(\bar{\Theta}-\Thetaa) = \mathcal{Q}(p_t)
\end{aligned} 
\right.	
\end{equation}
 which has been investigated recently by  Nikoli\'c and Said-Houari  in \cite{Nikolic_2022}  and \cite{NIKOLIC2022628}.    More precisely, for  Dirichlet--Dirichlet boundary conditions, they proved in \cite{Nikolic_2022}  a local well-posedness result of \eqref{Westervelt--Pennes--Fourier} by employing   the energy method together with a fixed point argument. In \cite{Nikolic_2022} the  global existence result and the asymptotic behavior of the solution of \eqref{Westervelt--Pennes--Fourier} under a smallness assumption on the initial data,  was established. Using the maximal regularity estimate for parabolic systems, Wilke in \cite{Wilke_2022} improves slightly the regularity assumptions in \cite{Nikolic_2022} and also considered the case $b=b(\bar{\Theta})$. We also mention the recent result in \cite{Nikolic_Said_Houari_Mixed_Cond} where a local  well-posedness result was shown for mixed Neumann and nonlinear absorbing boundary conditions in the acoustic component and Dirichlet boundary condition for the temperature.
 
 Recently, the authors of this paper considered in \cite{Benabbas_Said_Houar_2023} the Westervelt--Pennes--Cattaneo model where the second equation in \eqref{Westervelt--Pennes--Fourier} is replaced by its hyperbolic version (see  \cite[Eq.\ 3]{kabiri2021analysis} and \cite[Eq. 7]{xu2008non}): 
\begin{equation} \label{Hyperbolic_Pennes}
\begin{aligned}
\tau \rhoa \Ca &\bar{\Theta}_{tt}+(\rhoa \Ca +\tau \rhob \Cb W) \bar{\Theta}_t +\rhob \Cb W (\bar{\Theta}-\Thetaa)-\kappaa \Delta \bar{\Theta} \\
=&\,\mathcal{Q}(p_t)+ \tau \partial_t\mathcal{Q}(p_t).
\end{aligned}
\end{equation}
 Based on higher-order energy estimates and 
on the contraction mapping principle, we proved the existence and uniqueness of a strong solution of the Westervelt--Pennes--Cattaneo system. In addition,  using uniform estimates with respect to the relaxation parameter
$\tau$, we  showed  that
the solutions to the Westervelt--Pennes--Cattaneo problem converge to solutions of the
Westervelt--Pennes--Fourier system provided the relaxation parameter $\tau$ tends to zero.   
    
    In this paper, we consider the JMGT--Pennes system \eqref{coupled_problem_JMGT_Pennes} and establish its local well-posedness,   provided that the initial data satisfies
\begin{equation}
(p_0, p_1, p_2) \in H^2( \Om) \cap H^1_0(\Om) \times H^2( \Om) \cap H^1_0(\Om) \times H^1_0(\Om)\quad  \text{and}\quad  \Theta_0\in H^3(\Omega),
\end{equation}
 with a smallness condition on the pressure data that can be imposed on a lower-order norm than the one needed in the energy analysis, provided that the final time is short enough. If we require a smallness assumption on the full energy norm for the pressure data, then the final time can be made arbitrarily large. From the technical point of view, the smallness assumption on the lower-topology makes the proof more involved since some extra estimates of lower-order energy are needed (see Proposition  \ref{lemma_lower_norms}). These estimates should be properly factored out in the nonlinear estimates using the interpolation inequality \eqref{Agmon_Inequality}.      
      
    The rest of the paper is organized  as follows: Section \ref{Section_Prel} contains the necessary theoretical preliminaries, where we collect some theoretical results that will prove useful in the sequel. In Section \ref{Main_Result}, we introduce the general assumptions on the coefficients in system \eqref{coupled_problem_JMGT_Pennes} and state the main result of this paper. Section \ref{Section_Linearized_JMGT}  is devoted to the energy analysis of the JMGT equation with variable coefficients and we prove several energy estimates,  while Section  \ref{Section_Pennes_Equation}  treats the linearized Pennes equation. Although the two linearized problems in Sections \ref{Section_Linearized_JMGT} and \ref{Section_Pennes_Equation} are decoupled, the primary difficulty lies in maintaining consistency in the estimates for each linearized equation.   This is due to the  
     coupling in the nonlinear problem via the source term $\mathcal{Q}(\bar{\Theta},p_t)$ and via the coefficients (which depend on $\bar{\Theta}$) within the JMGT equation. Recognizing this nonlinear characteristic of the problem inherently suggests a correlation between the estimates obtained from the linearized problems.  
        In   Section \ref{Section_Well_Posdness}, we prove the local well-posedness of the nonlinear problem \eqref{coupled_problem_JMGT_Pennes}.
        
\section{Statement of the problem and main result}\label{Main_Result}
In this section, we state our local well-posedness result. Before that, we introduce the essential assumptions concerning the coefficients and the source terms in \eqref{coupled_problem_JMGT_Pennes}.

For clarity of exposition, we set 
 \begin{equation}
m=\rhoa \Ca\qquad \text{and}\qquad \ell=\rhob \Cb W,
\end{equation} 
and  make the change of variables 
$\Theta=\bar{\Theta}-\Theta_a$
in the temperature, then we can recast system \eqref{coupled_problem_JMGT_Pennes} as 
\begin{subequations}\label{Main_system_JMGT} 
\begin{equation}\label{modified_temp_eq}
\left\{ 
\begin{aligned}
&\tau p_{ttt}+(1-2 k (\Theta) p)p_{tt}-h (\Theta)\Delta p - \zeta (\Theta)\Delta p_t = 2 k (\Theta)\left(p_{t}\right)^2, \, &\text{in} \ \Omega \times (0,T),\\
&m\Theta_t -\kappaa \Delta \Theta + \ell \Theta = \Phi(\Theta, p_t), \, &\text{in} \ \Omega \times (0,T).
\end{aligned}
\right.
\end{equation} 
System \eqref{modified_temp_eq} is complemented with homogeneous boundary conditions
\begin{eqnarray} \label{homog_dirichlet}
p\vert_{\partial \Om}=0, \qquad \Theta\vert_{\partial \Om}=0
\end{eqnarray} 
and the initial conditions
\begin{eqnarray} \label{init_cond}
p|_{t=0}=p_0,\quad p_t|_{t=0}=p_1,\quad \quad p_{tt}|_{t=0}=p_2, \quad \Theta|_{t=0}=\Theta_0:=\bar{\Theta}_0-\Thetaa.
\end{eqnarray}
\end{subequations} 
The involved functions $h, \zeta, k$ and $\Phi$ are given in terms of $c, \delta, K$ and $\mathcal{Q}$ as follows
\begin{equation}\label{funct_k}
\begin{aligned}
&k (\Theta)=K(\Theta+\Thetaa)=\frac{1}{\rho c^2(\Theta+\Thetaa)} \beta_{\text{acou}},\qquad   h (\Theta)=c^2(\Theta+\Thetaa),\\
&\zeta(\Theta)= \delta(\Theta+\Thetaa), \quad \Phi(\Theta, p_t)=\phi(\Theta) (p_t)^2 \quad \text{with} \quad \phi(\Theta)=\frac{2 \beta(\Theta+\Thetaa)}{\rhoa h^2 (\Theta)}.
\end{aligned}
\end{equation}
In order to prove  the well-posedness result of \eqref{Main_system_JMGT}, we make the following assumptions  on the functions $h,\, \zeta,\, k $ and $\Phi$ and on the initial data. Note that the conditions on the medium parameters reflect a polynomial dependence on the temperature, since typically the speed of sound $c=c(\Theta)$ and the diffusivity of sound $\beta=\beta(\Theta)$ are determined using a least-square fit of experimental observations, see \cite{connor2002bio}.

\begin{assumption}[Assumptions on the coefficients and source terms] \label{Assumption_nonlinear}
We assume  that $h \in C^2(\mathbb{R})$ and  there exists $h_1>0$ such that
\begin{subequations}
\begin{equation} \label{bound_h}
h (s) \geq h_1, \quad \forall s \in \mathbb{R}.\tag{H1}
\end{equation}
Moreover, assume that there exist $\gamma_1 >0$ and $C>0$, such that
\begin{equation} \label{h''_assump}
\vert h ''(s) \vert \leq C (1+\vert s \vert^{\gamma_1}), \quad \forall s \in \mathbb{R}.\tag{H2}
\end{equation} 
Using Taylor's formula, we also have
\begin{equation} \label{h'_assump}
\vert h '(s) \vert \leq C (1+\vert s \vert^{1+\gamma_1}), \quad \forall s \in \mathbb{R}.\tag{H3}
\end{equation}
\end{subequations}
Since the function $k $ is related to the speed of sound by the formula \eqref{funct_k}, it follows that
\begin{subequations}
\begin{equation}\label{k_1}
\vert k (s) \vert \leq k_1:=\frac{\beta_{\text{acous}}}{\rho h_1}.\tag{K1}
\end{equation}
Further, we have 
\begin{equation}
\begin{aligned}
%\vert k '(s) \vert &\lesssim k_1^2\vert h '(s) \vert \lesssim k_1^2 (1+\vert s \vert^{1+\gamma_1}),\\
\vert k ''(s) \vert &\lesssim k_1^2 \vert h ''(s) \vert+k_1^3 \vert h '(s) \vert^2 \lesssim  k_1^2(1+\vert s \vert^{\gamma_1})+k_1^3(1+\vert s \vert^{1+\gamma_1})^2,
\end{aligned}
\end{equation}
which by using Taylor's formula, implies that there exists $\gamma_2>0$, such that
\begin{equation} \label{properties_k}
\vert k '(s) \vert \lesssim (1+\vert s \vert^{1+\gamma_2}), \qquad \vert k ''(s) \vert \lesssim (1+\vert s \vert^{\gamma_2}).\tag{K2}
\end{equation}
\end{subequations}
Similarly, since $\zeta$ is proportional to the speed of sound (see \eqref{funct_k}, \eqref{delta_definition}), we can assume that there exit $\zeta_1>0$ and $\gamma_3>0$ such that
\begin{subequations}
\begin{equation} \label{zeta}
\zeta(s) \geq \zeta_1, \quad \vert \zeta '(s) \vert \lesssim (1+\vert s \vert^{1+\gamma_3}), \qquad \vert \zeta ''(s) \vert \lesssim (1+\vert s \vert^{\gamma_3}), \qquad s \in \mathbb{R}. \tag{Z}
\end{equation} 
In addition, from the definition of the functions $\phi, \beta, h$ \eqref{funct_k}, \eqref{beta_definition}, we can see that 
$$ \phi(\Theta) \sim \frac{1}{c(\Theta)}.$$
Then, $\phi$ satisfies
\begin{equation} \label{phi_1}
\vert \phi(s) \vert \leq \phi_1, \quad s \in \mathbb{R}
\tag{F1}
\end{equation}
and as for the function $k$, there exists $\gamma_4>0$ such that 
\begin{equation} \label{phi_inequalities}
 \vert \phi '(s) \vert \lesssim (1+\vert s \vert^{1+\gamma_4}), \qquad \vert \phi ''(s) \vert \lesssim (1+\vert s \vert^{\gamma_4}), \qquad s \in \mathbb{R}.\tag{F2}
\end{equation}
\end{subequations}
\end{assumption}

\begin{assumption}[Assumptions on the initial data]\label{Assumption_Initial data}
 Let the initial data have the regularity
\begin{equation}
\begin{aligned}
 (p_0, p_1, p_2) \in&\, \big[H^2(\Om) \cap H^1_0(\Om) \big] \times  \big[ H^2(\Om) \cap H^1_0(\Om) \big] \times H^1_0(\Om)\\
\Theta_0 \in&\, H^3(\Om) \cap H^1_0(\Om),
\end{aligned}
\end{equation}
and let the compatibility condition be satisfied
\begin{equation}
%p_3:=p_{ttt}|_{t=0} \in L^2(\Om), \qquad 
\Theta_1=\Theta_t|_{t=0} \in H^1_0(\Om).
\end{equation}
\end{assumption}
\subsection{Main result}

The main contribution of this work pertains to the existence and uniqueness of a local in-time solution to the coupled system \eqref{Main_system_JMGT}. 
In order to state and prove our main result, we introduce the following spaces of solutions 
\begin{equation}\label{spaces}
\begin{aligned}
\mathcal{X}_p^2:=&\,\Big\{p \in L^{\infty}(0, T; H^2( \Om) \cap H^1_0(\Om)), p_t \in L^{\infty}(0, T; H^2( \Om) \cap H^1_0(\Om)),\\
&  \qquad p_{tt} \in L^{\infty}(0, T; H^1_0(\Om)) \Big\},\\
% \quad p_{ttt} \in L^{\infty}(0, T; L^2(\Om)) 
\mathcal{X}_\Theta^2:=&\,\big\{  \Theta \in L^2(0, T; H^3(\Om)\cap  H^1_0(\Om) )\cap L^\infty(0, T; H^2(\Om)\cap H^1_0(\Om)),\\
&\, \qquad \Theta_t \in L^\infty(0, T; H^1_0(\Om)) \cap L^2(0, T; H^2(\Om)\cap H^1_0(\Om)),\\
&\, \qquad \Theta_{tt} \in L^2(0, T; L^2(\Om)) \big\}.
\end{aligned}
\end{equation}
Now, we are in position to state the local well-posedness result for the problem \eqref{Main_system_JMGT}.
\begin{theorem}\label{local_Existence_Thoerem} Given $\tau>0$. Let Assumptions \ref{Assumption_nonlinear} and  \ref{Assumption_Initial data} hold. There exist a final time $T>0$,  small  $  \eta_0=\eta_0(T)>0$  and a constant $R>0$ such that if the initial data satisfy
\begin{equation}
\| p_0 \|_{H^1}^2 + \| p_1 \|_{H^1}^2 + \tau \| p_2 \|_{L^2}^2 \leq \eta_0^2,
\end{equation}
the system \eqref{Main_system_JMGT} admits a unique solution $(p, \Theta) \in \mathcal{X}_p^2 \times \mathcal{X}_\Theta^2$ that satisfies the estimate 
\begin{equation}
\begin{aligned}
&\| p \|_{L^\infty H^2}^2+ \| p_t \|_{L^\infty H^2}^2 +\tau \| p_{tt} \|_{L^\infty H^1}^2+ \| \Theta \|_{L^2 H^3}^2\\
& +\| \Theta \|_{L^\infty H^2}^2+ \| \Theta_t \|_{L^\infty H^1}^2+\| \Theta_t \|_{L^2 H^2}^2+ \| \Theta_{tt} \|_{L^2 L^2}^2 \leq R^2.
\end{aligned}
\end{equation}
Moreover, the solution depends continuously on the data with respect to the norm in $\mathcal{X}_p^2 \times \mathcal{X}_\Theta^2$.
\end{theorem}
The proof of Theorem \ref{local_Existence_Thoerem} will be given in Section \ref{Section_Well_Posdness}. But now we briefly make a few comments on the statement above. 
\begin{enumerate}
\item[1.]  We arrive at the conclusion of Theorem \ref{local_Existence_Thoerem} above by assuming a smallness condition on the pressure data on a  lower-order norm than the one needed in the energy analysis, provided that the final time is short enough. Indeed, the requirement for the pressure data to be small   arises naturally within the proof of the local well-posedness, and it is related to a nondegeneracy issue that characterizes nonlinear acoustic models as the one considered here, see \cite{Kaltenbacher_Nikolic_2022, Nikolic_2022, bongarti2021vanishing} and reference therein. %Otherwise, the final time $T$ being short is sufficient to %yield well-posedness. 
We point out that $T$ can be made arbitrarily large if the smallness assumption is imposed on the full energy norm of the pressure data. See the estimate \eqref{Inequalti_smallness_T}.  However, we prefer the former scenario since it is more relevant in ultrasound applications \cite{kaltenbacher2007numerical}, where usually the final time is short and the data are smooth but not small. As for the bioheat pennes equation, no smallness assumption is required on the temperature data.  
%\item[2.] By using time weight, the time integrability of the solution can be improved. For instance, it is possible to show,  under the same assumptions on the initial data, we can prove that $\sqrt{t}\,\Theta_t \in L^\infty(0, T; H^2(\Om)\cap  H^1_0(\Om) )$. See the recent publication \cite{}. 
\item [2. ]   It is known  (see for instance \cite{Kal_Las_Mar}) that for the JMGT equation \eqref{MGT_1} there is a critical parameter 
        \begin{equation}\label{chi}
       \chi=\delta -\tau c^{2},
       \end{equation}
       that controls  the behavior of the solution. More precisely, the authors in \cite{Kal_Las_Mar} proved  that the problem \eqref{MGT_1} is globally well-posed and its solution is exponentially stable if   $\chi>0$. Although,  this condition seems necessary to study the global existence and asymptotic behavior of the solution of \eqref{MGT_1} and \eqref{MGT_1_1},  it is unnecessary for the local existence theory as in the present paper (see also \cite{KaltenbacherNikolic}).

\end{enumerate}

\section{Preliminaries}\label{Section_Prel}
\subsection{Notation}
Throughout the paper, we assume that $\Omega \subset \R^d$, where $d \in \{1,2,3\}$, is a bounded domain with a $C^3$ boundary. We denote by $T>0$ the final propagation time. The letter $C$ denotes a generic positive constant
that does not depend on time, and can have different values on different occasions.  
We often write $f \lesssim g$ (resp. $f \gtrsim g$) when there exists a constant $C>0$, independent of parameters of interest such that $f\leq C g$ (resp. $f \geq C g$). 
%We note that dependance of a constant on the medium parameters is not always made explicit; however, we make sure to emphasize dependance on the relaxation time $\tau$. 
We often omit the spatial and temporal domain when writing norms; for example, $\|\cdot\|_{L^p L^q}$ denotes the norm in $L^p(0,T; L^q(\Omega))$.
\subsection{Inequalities and embedding results} In the upcoming analysis, we shall employ the continuous embeddings $H^1(\Om) \hookrightarrow L^4(\Om)$, $H^2(\Om) \hookrightarrow L^\infty(\Om)$. In particular, using Poincar\'{e}'s inequality we obtain for $v \in H^1_0(\Om)$ (see \cite[Theorem 7.18]{salsa2016partial})
\begin{equation}\label{Sobolev_Embedding}
\begin{aligned}
& \text{if} \ d>2, \quad \Vert v \Vert_{L^p} \leq C \Vert \nabla v \Vert_{L^2}  \quad \text{for} \quad  2 \leq p \leq \frac{2d}{d-2},\\
& \text{if} \ d=2, \quad \Vert v \Vert_{L^p} \leq C \Vert \nabla v \Vert_{L^2}  \quad \text{for} \quad  2 \leq p < \infty.\\
\end{aligned}
\end{equation}
Moreover, taking into account the boundedness of the operator $(-\Delta)^{-1}: L^2(\Om) \rightarrow H^2(\Om) \cap H^1_0(\Om)$, we obtain  the inequality
$$\qquad \Vert v \Vert_{L^\infty} \leq C_1 \Vert v \Vert_{H^2} \leq C_2 \Vert \Delta v \Vert_{L^2}.$$ 
%We will also call on the 1D-embedding $H^1(0, T, X) \hookrightarrow C(0, T; X)$; that is to say,  the inequality 
%\begin{equation} \label{1D_embedding}
%\begin{aligned}
%\max_{t\in [0, T]} \Vert v(t) \Vert_{X} &\leq C(\Vert v \Vert_{L^2 X}+\Vert v_t \Vert_{L^2 X})\\
%%& \leq C(\Vert \nabla v \Vert_{L^2 L^2}+\Vert v_t \Vert_{L^2 L^2}),
%\end{aligned}    
%\end{equation}
%holds  for all $v \in H^1(0, T; X)$, 
%where the constant $C>0$ depends only on $T$ (see, e.g. \cite[Theorem 2, p. 286]{evans2010partial}).\\
We will also make a repeated use of  Young's $\varepsilon$-inequality 
\begin{equation}\label{Young_Inequality}
xy \leq \varepsilon x^n+C(\varepsilon) y^m, \quad \text{where}\quad \ x, y >0, \quad 1 <m,n <\infty,\quad \frac{1}{m}+\frac{1}{n}=1,
\end{equation}
and $C(\varepsilon)=(\varepsilon n)^{-m/n}m^{-1}$.\\
One particular form of \eqref{Young_Inequality} that we frequently use in the proof is the following:
\begin{equation}
xy\leq \varepsilon x^2+\frac{1}{4\varepsilon} y^2. 
\end{equation}
We will also need Agmon's interpolation inequality~\cite[Ch.\ 13]{agmon2010lectures} for functions in $H^2(\Omega)$:
\begin{equation}\label{Agmon_Inequality}
\| u\|_{L^\infty(\Omega)} \leq C_{\textup{A}} \|u\|_{H^2(\Omega)}^{d/4}\|u\|_{L^2(\Omega)}^{1-d/4}, \qquad d \leq 4
\end{equation}
and Ladyzhenskaya's inequality for functions in $u \in H^1(\Om)$:
\begin{equation}\label{Lady_Inequality}
\| u\|_{L^4(\Omega)} \leq C_{\textup{A}} \|u\|_{H^1(\Omega)}^{d/4}\|u\|_{L^2(\Omega)}^{1-d/4}, \qquad d \leq 4.
\end{equation}
%Further, we will make use of  Ladyzhenskaya's inequality for $u \in H^1(\Om)$
%\begin{equation} \label{lady}
%\Vert u \Vert_{L^4} \leq C\Vert u \Vert_{L^2}^{1-d/4} \Vert u \Vert_{H^1}^{d/4},\qquad 1\leq d\leq 4.  
%\end{equation}   

 To keep the presentation self-contained, we also state here the precise version of Gronwall's inequality that will be employed in the proofs.
\begin{lemma}\label{Lemma_Gronwall}
We assume that $u\in C([0,T]; \R)$, $T\in (0,\infty)$, satisfies the integral inequality 

\begin{equation}
u(t)\leq u_0+\int_0^t a(s) u(s)\ds+\int_0^t b(s)\ds, \quad \text{on}\quad [0,T),
\end{equation}  
for some $0\leq a\in L^1(0,T)$ and $b\in L^1(0,T)$. Then it holds that 
\begin{equation}
u(t)\leq u_0 e^{A(t)}+\int_0^t b(s) e^{A(t)-A(s)}\ds,\quad \forall t\in (0,T),
\end{equation}
where
\begin{equation}
A(t)= \int_{0}^t a(s)\ds. 
\end{equation}
\end{lemma}
\section{ Analysis of the linearization of the JMGT equation}
\label{Section_Linearized_JMGT}
The method of the proof of the local existence is based on the application of the contraction mapping theorem to solutions of a related linear problem. For this reason, we need to analyze two linear (decoupled)  systems, one associated with the JMGT equation and the other one with the Pennes model. First, we  consider the following linearization of the first equation in system \eqref{Main_system_JMGT}:
\begin{equation}\label{Linearized_Eq_JMGT}
\left\{
\begin{aligned}
&\tau p_{ttt}+ p_{tt}-r(x,t) \Delta p -b(x,t) \Delta p_t=f(x,t), \qquad (x,t) \in \Om \times (0, T)\\
& p|_{t=0}=p_0,\quad p_t|_{t=0}=p_1,\quad p_{tt}|_{t=0}=p_2,\\
&p|_{\partial\Omega}=0.
\end{aligned}
\right.
\end{equation}
We are able to establish the well-posedness of \eqref{Linearized_Eq_JMGT} under the following regularity and  non-degeneracy  conditions on the coefficients $b$ and $ r$.
\begin{assumption}\label{Assumption_1}
Assume that
\begin{itemize}
\item $r, \,b \in L^{\infty}(0, T; L^\infty (\Om))\cap L^2(0, T; W^{1, \infty}(\Om))$ and  $r_t,\, b_t \in L^2(0, T;L^\infty(\Om))$, 
%\item $b \in L^{\infty}(0, T; L^\infty (\Om)) \cap L^2(0, T; W^{1, \infty}(\Om))$, $b_t \in L^2(0, T;L^\infty(\Om))$,
\item There exist $0<b_0\leq b_1$ and $0 <r_0 \leq r_1$ such that 
\begin{equation}
\begin{aligned}
& b_0 \leq b(x, t) \leq b_1 \quad \text{a.e. in}\quad \Om \times (0, T); \\
& r_0 \leq r(x, t) \leq r_1 \quad \text{a.e. in} \quad \Om \times (0, T).
\end{aligned}
\end{equation}
%\item Let the compatibility condition hold
%\begin{equation}
%p_3:=p_{ttt}|_{t=0} \in L^2(\Om)
%\end{equation}
%where $p_3$ is defined as
%\begin{equation}
%p_3(x)=\frac{1}{\tau} \big(-p_2+r(x,0) \Delta p_0 +b(x,0) \Delta p_1 + f(x,0) \big), \quad x \in \Om.
%\end{equation}
\end{itemize}  
\end{assumption}
In what follows, we prove two results on the existence of solution of the linear problem \eqref{Linearized_Eq_JMGT}. Precisely, we establish well-posedness respectively for initial data $(p_0,p_1,p_2)\in H^1_0(\Om) \times H^1_0(\Om) \times L^2(\Om)$ (Proposition \ref{lemma_lower_norms}) and for $(p_0,p_1,p_2) \in \big[H^2(\Om) \cap H^1_0(\Om)\big] \times \big[H^2(\Om) \cap H^1_0(\Om)\big] \times H^1_0(\Om)$ (Proposition \ref{prop1}).
 This is achieved by employing the energy analysis for smooth Galerkin approximations of the solutions. We will focus our attention on providing the  energy bounds that will later be useful when we address the well-posedness of the nonlinear problem \eqref{Main_system_JMGT}. We omit the other steps in Galerkin method since they are quite classical and are already exposed in details in the context of the JMGT equation in \cite{kaltenbacher2019vanishing}. We note that in view of the hyperbolic nature of the equation \eqref{Linearized_Eq_JMGT} and the time-space variability of the coefficients of the higher-order derivatives, it is crucial that $r_t,\, b_t$ belong to $L^\infty(\Om)$. It does not seem feasible that such  assumption can be weakened even if we attempt to counteract by increasing the regularity of the initial data.
\subsection{First Energy estimate}
In the first functional setting, the energies $E_0, E_1$ are defined by
\begin{equation}
\begin{aligned}
&E_0[p](t):= \frac{1}{2} \Big(\| p_t(t) \|_{L^2}^2+ \| \sqrt{r} \nabla p(t) \|_{L^2}^2 \Big),\\
&E_1[p](t):= \frac{1}{2} \Big(\tau \| p_{tt}(t) \|_{L^2}^2+ \| \sqrt{b} \nabla p_t(t) \|_{L^2}^2 \Big), \qquad t \geq 0.
\end{aligned}
\end{equation}   
The dissipation rate of the total energy $E_0+E_1$ is 
\begin{equation}
\begin{aligned}
&D_0[p](t):= \| \sqrt{b} \nabla p_t(t) \|_{L^2}^2+  \| p_{tt}(t) \|_{L^2}^2, \qquad t \geq 0.
\end{aligned}
\end{equation}
We define  the space of solutions $\mathcal{X}_p^1$ as  
\begin{equation}\label{space_p_1}
\begin{aligned}
\mathcal{X}_p^1:=&\,\Big\{p \in L^{\infty}(0, T; H^1_0(\Om)),\quad  p_t \in L^{\infty}(0, T; H^1_0(\Om)),\\
& \qquad \qquad p_{tt} \in L^{\infty}(0, T; L^2(\Om))\Big\}.
\end{aligned}
\end{equation}
\begin{proposition}\label{lemma_lower_norms} 
Let $T >0, \, \tau >0$. Assume that $f  \in L^2(0, T; L^2(\Om))$ and that the functions $r$ and $b$ satisfy Assumption \ref{Assumption_1}. Given
$$(p_0, p_1, p_2) \in H^1_0(\Om) \times H^1_0(\Om) \times L^2(\Om), $$
then, the problem \eqref{Linearized_Eq_JMGT} has a unique solution $p \in \mathcal{X}_p^1$. Moreover, there exists a constant $C_{p, 1}(\tau)>0$, that does not depend on $T$, such that for all $0 \leq t  \leq T$ it holds
\begin{equation}\label{Main_First_Order_Est}
\begin{aligned}
& E_0[p](t)+E_1[p](t)+ \int_0^t D_0(s) \ds\\
 \leq  &\, C_{p, 1} \Big[ 
\Big(E_0[p](0)+E_1[p](0)\Big)\exp\left\{\int_0^t \Lambda_0(s)  \ds \right\}\\
&\, +\int_0^t F_0(s) \exp\left\{\int_s^t \Lambda_0(r)  \ \textup{d} r \right\} \ds\Big]. 
\end{aligned}
\end{equation}
where 
\begin{equation} \label{lambda_0}
\Lambda_0(t)=1+ \| r_t(t)\|_{L^\infty}+ \| r_t(t) \|_{L^\infty}^2+\| b_t(t) \|_{L^\infty} +\| \nabla b(t) \|^2_{L^\infty}+\| \nabla r(t) \|^2_{L^\infty},
\end{equation}
and 
\begin{equation}
F_0(t) = \| f(t) \|_{L^2}^2.
\end{equation}
\end{proposition}
\begin{proof}
We multiply the equation \eqref{Linearized_Eq_JMGT} by $p_t$ and integrate over $\Om$, using  integration by parts, we get
\begin{equation}
\begin{aligned}
&\frac{1}{2} \ddt \Big( \|p_t \|_{L^2}^2 +\| \sqrt{r} \nabla p \|_{L^2}^2 \Big)+\| \sqrt{b} \nabla p_t \|_{L^2}^2\\
=&\, -\ddt ( \intO \tau p_{tt} p_t \dx)+ \tau \| p_{tt} \|_{L^2}^2- \intO \nabla b \cdot \nabla p_t p_t \dx-\intO \nabla r \cdot \nabla p p_t \dx \\
&\,+ \frac{1}{2} \intO r_t \vert \nabla p \vert^2 \dx+ \intO f p_t  \dx,
\end{aligned}
\end{equation}
where we have used the identity
\begin{equation}
\intO \tau p_{ttt} p_t \dx =\ddt \big(\intO \tau p_{tt} p_t \dx \big)-\tau \| p_{tt} \|^2_{L^2}.
\end{equation}
Applying H\"{o}lder and Young inequalities yields
\begin{equation}\label{E_1_1}
\begin{aligned}
&\frac{1}{2} \ddt \Big( \|p_t \|_{L^2}^2 +\| \sqrt{r} \nabla p \|_{L^2}^2 \Big)+\| \sqrt{b} \nabla p_t \|_{L^2}^2\\
\leq &\, -\ddt ( \intO \tau p_{tt} p_t \dx)+\tau \| p_{tt} \|_{L^2}^2 + \frac{1}{2} \| r_t\|_{L^\infty} \| \nabla p \|_{L^2}^2\\
&\,+ C(\varepsilon)( \| \nabla b \|_{L^\infty}^2 \| \nabla p_t \|^2_{L^2}+\| \nabla r \|_{L^\infty}^2 \| \nabla p \|^2_{L^2} + \| f \|_{L^2}^2)+ \varepsilon \| p_t \|_{L^2}^2.
\end{aligned}
\end{equation}
Since $p_t|_{\partial \Om}=0$, we can employ  Poincar\'{e}'s inequality to get 
\begin{equation}\label{Poincare_1}
\| p_t \|_{L^2}^2\lesssim  \left\|\frac{1}{b}\right\|_{L^\infty}\|\sqrt{b} \nabla p_t \|_{L^2}^2\lesssim \|\sqrt{b} \nabla p_t \|_{L^2}^2
\end{equation}
by using Assumption \ref{Assumption_1}. Plugging   \eqref{Poincare_1} into \eqref{E_1_1} and 
 choosing  $\varepsilon$ small enough,  we obtain
\begin{equation}
\begin{aligned}
& \ddt E_0[p](t)+\| \sqrt{b} \nabla p_t(t) \|_{L^2}^2\\
\lesssim &\,(\tau+ \| r_t\|_{L^\infty}+ \| \nabla r \|_{L^\infty}^2+ \| \nabla b \|_{L^\infty}^2)( \| \nabla p \|_{L^2}^2+\| \nabla p_t \|_{L^2}^2+\| p_{tt} \|_{L^2}^2) 
  \\
&+\| f \|_{L^2}^2-\ddt \Big( \intO \tau p_{tt} p_t \dx\Big)
\end{aligned}
\end{equation}
Integrating from $0$ to $t$ and keeping Assumption \ref{Assumption_1} in mind, it follows
\begin{equation} \label{ineq1}
\begin{aligned}
& E_0[p](t)+\int_0^t\| \sqrt{b} \nabla p_t(s) \|_{L^2}^2\ds\\
\lesssim &\, E_0[p](0)+ \intO \tau \vert p_{tt}(t) p_t(t) \vert \dx+ \intO \tau \vert p_2 p_1 \vert \dx \\
&\, +\int_0^t(\tau+ \| r_t\|_{L^\infty}+\| \nabla r \|_{L^\infty}^2+\| \nabla b \|_{L^\infty}^2)( E_0[p](s) +E_1[p](s)) \ds \\
&+\int_0^t\| f(s) \|_{L^2}^2 \ds.
\end{aligned}
\end{equation}
Next, we have for any $\varepsilon>0$, 
\begin{equation} \label{ineq2}
\begin{aligned}
&\intO \tau \vert p_{tt}(t)  p_t(t) \vert \dx \leq C(\varepsilon)\tau\|p_{tt}(t) \|_{L^2}^2 + \varepsilon \tau \|p_{t}(t) \|_{L^2}^2, \quad \forall t \geq 0,\\
&\intO \tau \vert p_2 p_1 \vert \dx \leq \frac{\tau}{2}\|p_2 \|_{L^2}^2+\frac{\tau}{2}\|p_1 \|_{L^2}^2.
\end{aligned}
\end{equation}
By inserting the estimates in \eqref{ineq2} into  \eqref{ineq1} and selecting $\varepsilon$ suitably small, we obtain
\begin{equation} \label{ineq6}
\begin{aligned}
& E_0[p](t)+\int_0^t\| \sqrt{b} \nabla p_t(s) \|_{L^2}^2\ds\\
\lesssim &\, E_0[p](0)+E_1[p](0)+  \tau\|p_{tt}(t) \|_{L^2}^2 \\
&\, +\int_0^t\Lambda_0(s)( E_0[p](s)+E_1[p](s)) \ds +\int_0^t\| f(s) \|_{L^2}^2 \ds.
\end{aligned}
\end{equation}
Now, we multiply the equation \eqref{Linearized_Eq_JMGT} by $p_{tt}$ and we integrate over $\Om$ to arrive at the identity
\begin{equation}\label{dE_1_dt}
\begin{aligned} 
& \ddt E_1[p](t)+\|p_{tt} \|_{L^2}^2\\
=&\, - \intO \nabla r \cdot \nabla p p_{tt} \dx- \intO r \nabla p \cdot \nabla p_{tt} \dx +\frac{1}{2} \intO b_t \vert \nabla p_t \vert^2 \dx\\
&\,-\intO \nabla b \cdot \nabla p_t p_{tt} \dx + \intO f p_{tt} \dx.
\end{aligned}
\end{equation}
Our next goal is to estimate the right-hand side of \eqref{dE_1_dt}. 
Applying   H\"{o}lder and Young inequalities, we find
\begin{equation}
\begin{aligned}
&\ddt E_1[p](t)+\|p_{tt} \|_{L^2}^2\\
\leq &\, - \intO r \nabla p \cdot \nabla p_{tt} \dx +\frac{1}{2} \| b_t \|_{L^\infty} \Vert \nabla p_t \Vert^2_{L^2}\\
&\,+C(\varepsilon)(\| \nabla r \|^2_{L^\infty} \|\nabla p \|^2_{L^2}+\| \nabla b \|^2_{L^\infty} \|\nabla p_t \|^2_{L^2}+ \| f \|^2_{L^2}) +\varepsilon  \| p_{tt} \|^2_{L^2}.
\end{aligned}
\end{equation}
Hence, integrating in time and taking account of Assumption \ref{Assumption_1},  we have for $\varepsilon$ small enough  
\begin{equation} \label{ineq3}
\begin{aligned}
& E_1[p](t)+\int_0^t\|p_{tt}(s) \|_{ L^2}^2\ds\\
\lesssim &\,E_1[p](0) -\int_0^t \intO  r \nabla p \cdot \nabla p_{tt} \dx \ds\\
&\,+\int_0^t(\| b_t \|_{L^\infty} +\| \nabla b \|^2_{L^\infty}+\| \nabla r \|^2_{L^\infty})( E_0[p](s)+E_1[p](s)) \ds\\
&+ \int_0^t\| f(s) \|_{L^2}^2 \ds.
\end{aligned}
\end{equation}
The second term  on the right-hand side  of \eqref{ineq3} can be handled as follows
\begin{equation}
\begin{aligned}
& -\intO  r \nabla p \cdot \nabla p_{tt} \dx\\
=&\, -\ddt \left( \intO  r \nabla p \cdot \nabla p_{t} \dx\right)+\intO r_t \nabla p \cdot \nabla p_t \dx+ \| \sqrt{r} \nabla p_t \|^2_{L^2}\\
\leq &\, -\ddt \left( \intO  r \nabla p \cdot \nabla p_{t} \dx\right) +\frac{1}{2} \| r_t \|_{L^\infty}^2 \| \nabla p \|_{L^2}^2 + \frac{1}{2}\| \nabla p_t \|_{L^2}^2+ \| \sqrt{r} \nabla p_t \|^2_{L^2}.
\end{aligned}
\end{equation}
Integrating over time and exploiting   Assumption \ref{Assumption_1}, it results that
\begin{equation}
\begin{aligned}
-\int_0^t\intO  r \nabla p \cdot \nabla p_{tt} \dx \ds
\lesssim &\, \intO  \vert r(0) \nabla p_0 \cdot \nabla p_{1} \vert \dx+\intO  \vert r(t) \nabla p(t) \cdot \nabla p_{t}(t) \vert \dx\\ &\,+ \int_0^t(1+ \| r_t \|_{L^\infty}^2)(\| \nabla p \|_{L^2}^2 + \| \nabla p_t \|_{L^2}^2) \ds.
\end{aligned}
\end{equation}
Using H\"older and Young inequalities, we can estimate the first two integrals on the right-hand side of the above inequality to find 
\begin{equation} \label{ineq4}
\begin{aligned}
-\int_0^t\intO  r \nabla p \cdot \nabla p_{tt} \dx
\lesssim &\, \| \nabla p_0 \|^2_{L^2}+\| \nabla p_{1} \|^2_{L^2}\\&+ \|\nabla p(t) \|^2_{L^2}+ \varepsilon\left\| \frac{r(t)}{\sqrt{b(t)}} \right\|^2_{L^\infty} \| \sqrt{b(t)} \nabla p_{t}(t)\|^2_{L^2} \\
&\, + \int_0^t(1+ \| r_t \|_{L^\infty}^2) ( E_0[p](s)+E_1[p](s)) \ds.
\end{aligned}
\end{equation}
Plugging inequality \eqref{ineq4} into \eqref{ineq3} and taking $\varepsilon$ small enough so as to get the $\varepsilon$-term in \eqref{ineq4} absorbed by the left-hand side of \eqref{ineq3}, we obtain 
\begin{equation}  \label{ineq5}
\begin{aligned}
&E_1[p](t)+\int_0^t\|p_{tt}(s) \|_{ L^2}^2\ds\\
\lesssim &\, E_0[p](0)+E_1[p](0)+\int_0^t\Lambda_0(s)( E_0[p](s)+E_1[p](s)) \ds\\
&\,+\|\nabla p(t) \|^2_{L^2}+ \int_0^t\| f(s) \|_{L^2}^2 \ds.
\end{aligned}
\end{equation}
Observe that for $0 \leq t \leq T$, we have 
\begin{equation}\label{First_Term}
\begin{aligned}
\|\nabla p(t) \|^2_{L^2} \lesssim &\,  \|\nabla p_0 \|^2_{L^2} +\int_{0}^t \|\sqrt{b}\nabla p_t(s) \|^2_{L^2}\ds\\
\lesssim &\, E_0[p](0)+\int_{0}^t E_1[p](s)\ds. 
\end{aligned}
\end{equation}
Hence, combining \eqref{First_Term} with \eqref{ineq5}, we have  
\begin{equation} \label{ineq_5_Main}
\begin{aligned}
&E_1[p](t)+\int_0^t\|p_{tt}(s) \|_{ L^2}^2\ds\\
\lesssim &\, E_0[p](0)+E_1[p](0)
+\int_0^t\Lambda_0(s)( E_0[p](s)+E_1[p](s)) \ds+ \int_0^t\| f(s) \|_{L^2}^2 \ds.
\end{aligned}
\end{equation}
Now with the aim of absorbing  the term $\tau \| p_{tt} \|^2_{L^2}$ from the right-hand side of \eqref{ineq6},  we add $N_1 \times$\eqref{ineq_5_Main} to  \eqref{ineq6} and select  $N_1$ large enough (i.e., $\frac{N_1}{2}\gtrsim 1$), we conclude  
\begin{equation} \label{First_Energy_Estimate_Main}
\begin{aligned}
& E_0[p](t)+E_1[p](t)+\int_0^t D_0[p](s)\ds\\
\lesssim &\, E_0[p](0)+E_1[p](0)
 +\int_0^t\Lambda_0(s)( E_0[p](s)+E_1[p](s)) \ds +\int_0^t\| f(s) \|_{L^2}^2 \ds.
\end{aligned}
\end{equation}
Applying the Gronwall's inequality (Lemma \ref{Lemma_Gronwall}) yields the desired result \eqref{Main_First_Order_Est}. 
\end{proof}
\subsection{Higher-order energy estimate}
Due to the coupling between the pressure equation and the bioheat equation being via the coefficients in the former and the source term in the latter, the regularity of the acoustic pressure  guaranteed by Proposition  \ref{lemma_lower_norms} is not enough to give the temperature $\Theta$ the right smoothness so that Assumption \ref{Assumption_1} will be meaningful when we get to the study of the nonlinear problem \eqref{Main_system_JMGT}. Therefore, we need to show the existence of more regular solutions to the problem \eqref{Linearized_Eq_JMGT}. This will be done again by relying on Galerkin's smooth approximations of the solution  and on higher-order energy estimates  as in \cite{KaltenbacherNikolic}. \\
For this purpose, we define the higher-order  energy associated with the equation \eqref{Linearized_Eq_JMGT} by
\begin{equation}
\begin{aligned}  
&E[p](t):=\frac{1}{2} \big(\Vert \sqrt{b} \Delta p \Vert^2_{L^2}+\Vert \sqrt{b} \Delta p_t \Vert^2_{L^2} +\tau \Vert \nabla p_{tt} \Vert_{L^2}^2\big).
\end{aligned}
\end{equation}
We also introduce the corresponding dissipation
\begin{equation}
D[p](t):= \Vert \sqrt{r} \Delta p \Vert^2_{L^2}+\Vert  \nabla p_{tt} \Vert^2_{L^2}.
\end{equation}
%We denote by $\mathcal{X}_p^2$ the functional space defined as
%\begin{equation}\label{}
%\begin{aligned}
%\mathcal{X}_p^2:=&\,\Big\{p \in L^{\infty}(0, T; H^2( \Om) \cap H^1_0(\Om)), p_t \in L^{\infty}(0, T; H^2( \Om) \cap H^1_0(\Om)),\\
%& \quad p_{tt} \in L^{\infty}(0, T; H^1_0(\Om)), p_{ttt} \in L^{\infty}(0, T; L^2(\Om)) \Big\}.
%\end{aligned}
%\end{equation}
We prove the following existence result.
\begin{proposition} \label{prop1}
Let $\tau >0$ and $T>0$. Assuming that $f \in L^2(0, T; H^1_0(\Om))$  and that  Assumption \ref{Assumption_1} holds.  Suppose that   
\begin{equation}
(p_0, p_1, p_2) \in \big[H^2( \Om) \cap H^1_0(\Om)\big] \times \big[H^2( \Om) \cap H^1_0(\Om)\big] \times H^1_0(\Om). 
\end{equation}
Then,  the system \eqref{Linearized_Eq_JMGT} admits a unique solution $p \in \mathcal{X}_p^2$ (see \eqref{spaces}). Moreover, there exists a constant $C_{p, 2}(\tau)>0$,  that does not depend on $T$, such that the solution satisfies the following estimate for $0 \leq t \leq T$ 
\begin{equation} \label{main_estimate_E_p}
\begin{aligned}
& E[p](t)+\int_0^t D[p](s) \ds \\
\leq  &\,  C_{p, 2} \Big[E[p](0)\exp\Big(\int_0^t \Lambda(s) \ds \Big)+ \int_0^t F(s) \exp\Big(\int_s^t \Lambda(r) \textup{d} r \Big) \ds \Big]
\end{aligned}
\end{equation}
where $\Lambda$ is defined as 
\begin{equation} \label{lambda}
\Lambda(t)= 1+ \Vert b_t (t) \Vert_{L^\infty}+\Vert b_t(t) \Vert_{L^\infty}^2+ \Vert r_t(t) \Vert_{L^\infty}^2
\end{equation}
and 
$$F(t)= \Vert f(t) \Vert^2_{L^2}+\Vert \nabla f(t) \Vert^2_{L^2}.$$
\end{proposition}
\begin{proof} 
We begin by multiplying the equation \eqref{Linearized_Eq_JMGT} by $-\Delta p$, and integrating over $\Om$, using, integration by parts, we  find 
\begin{equation} \label{identity_3}
\begin{aligned}
\frac{1}{2} \ddt \Vert \sqrt{b}\Delta p \Vert^2_{L^2}+\Vert \sqrt{r} \Delta p \Vert^2_{L^2}=&\, \intO (p_{tt}-f)  \Delta p \dx+ \ddt \Big( \intO  \tau p_{tt} \Delta p \dx \Big)\\
&-\intO \tau p_{tt} \Delta p_t \dx+\intO b_t \vert \Delta p \vert^2\dx, 
\end{aligned}
\end{equation}
where we have used the identity
\begin{equation}
\intO \tau p_{ttt} \Delta p \dx= \ddt \big( \intO  \tau p_{tt} \Delta p \dx \big)-\intO \tau p_{tt} \Delta p_t \dx.
\end{equation}
In order to estimate the right-hand side of \eqref{identity_3}, we first use H\"{o}lder and Poincar\'{e} inequalities to obtain
\begin{equation} 
\begin{aligned}
&\frac{1}{2} \ddt \Vert \sqrt{b}\Delta p \Vert^2_{L^2}+\Vert \sqrt{r} \Delta p \Vert^2_{L^2}\\
\leq &\,  \ddt \Big( \intO  \tau p_{tt} \Delta p \dx \Big)+ \big(\Vert \nabla  p_{tt} \Vert_{L^2}+ \Vert b_t \Vert_{L^\infty}  \Vert \Delta p \Vert_{L^2}+ \Vert f  \Vert_{L^2} \big) \Vert  \Delta p \Vert_{L^2}\\
&+\tau \Vert \nabla p_{tt} \Vert_{L^2} \Vert  \Delta p_t \Vert_{L^2}. 
\end{aligned}
\end{equation}
From here, by applying Young's inequality, we find
\begin{equation} 
\begin{aligned}
&\frac{1}{2} \ddt \Vert \sqrt{b}\Delta p \Vert^2_{L^2}+\Vert \sqrt{r} \Delta p \Vert^2_{L^2}\\
\leq &\,   \ddt \Big( \intO  \tau p_{tt} \Delta p \dx \Big)+C(\varepsilon) \Big(\Vert \nabla  p_{tt} \Vert_{L^2}^2+\Vert b_t \Vert_{L^\infty}^2 \Big\Vert \frac{1}{\sqrt{b}} \Big\Vert^2_{L^\infty}  \Vert \sqrt{b}\Delta p \Vert^2_{L^2}+\Vert f  \Vert_{L^2}^2 \Big)\\
&+\varepsilon
\Big\Vert \frac{1}{\sqrt{r}} \Big\Vert^2_{L^\infty}\Vert \sqrt{r} \Delta p \Vert_{L^2}^2 + \frac{\tau}{2} \Big\Vert \frac{1}{\sqrt{b}} \Big\Vert^2_{L^\infty}  \Vert \sqrt{b} \Delta p_t \Vert_{L^2}^2+ \frac{\tau}{2} \Vert \nabla p_{tt} \Vert_{L^2}^2 . 
\end{aligned}
\end{equation}
By selecting $\varepsilon $ small enough,  the $\varepsilon $-term on the right-hand side  can be absorbed by  the dissipative term on the left-hand side. Then, on account of Assumption \ref{Assumption_1}, we have 
\begin{equation} 
\begin{aligned}
&\frac{1}{2} \ddt \Vert \sqrt{b}\Delta p(t) \Vert^2_{L^2}+\Vert \sqrt{r} \Delta p(t) \Vert^2_{L^2} \\
\lesssim &\, \ddt \big( \intO  \tau p_{tt} \Delta p \dx \big)+\big( 1+ \frac{1}{\tau}+ \tau+\Vert b_t \Vert_{L^\infty}^2 \big)E[p](t)+\Vert f (t) \Vert_{L^2}^2.
\end{aligned}
\end{equation}
Integrating from $0$ to $t$ yields the following bound for $0 \leq t \leq T$
\begin{equation} \label{estimate_delta_p} 
\begin{aligned}
&\frac{1}{2}  \Vert \sqrt{b(t)}\Delta p(t) \Vert^2_{L^2}+ \int_0^t\Vert \sqrt{r(s)} \Delta p(s) \Vert^2_{L^2} \ds\\
\lesssim &\,  \intO  \tau \vert p_{tt}(t) \Delta p(t) \vert \dx+\intO  \tau \vert p_{2} \Delta p_0 \vert \dx + \Vert \sqrt{b(0)}\Delta p_0 \Vert^2_{L^2}\\
&\,+\int_0^t \big( 1+ \tau+ \frac{1}{\tau}+\Vert b_t(s) \Vert_{L^\infty}^2\big)E[p](s) \ds+ \int_0^t \Vert f(s) \Vert_{ L^2}^2 \ds.
\end{aligned}
\end{equation}
We can use H\"{o}lder and Young inequalities to bound the first two terms on the right-hand side as follows:
\begin{equation} \label{est_1} 
\begin{aligned}
&\,\intO  \tau \vert p_{tt}(t) \Delta p(t) \vert \dx+\intO  \tau \vert p_{2} \Delta p_0 \vert \dx \\
\leq &\, C(\varepsilon) \tau^2\| p_{tt}(t) \|_{L^2}^2+ \varepsilon \Big\Vert \frac{1}{\sqrt{b}} \Big\Vert^2_{L^\infty}\| \sqrt{b} \Delta p(t) \|_{L^2}^2 + \frac{\tau}{2} \|p_{2} \|_{L^2}^2\\
&+ \frac{\tau}{2} \Big\Vert \frac{1}{\sqrt{b}} \Big\Vert^2_{L^\infty} \| \sqrt{b(0)}\Delta p_0 \|_{L^2}^2.
\end{aligned}
\end{equation}
Due to Assumption \ref{Assumption_1}, we can combine \eqref{estimate_delta_p} and \eqref{est_1}  and use Poincar\'{e}'s inequality to obtain, for small enough $\varepsilon$, the following bound
\begin{equation} \label{estimate_delta_p_1} 
\begin{aligned}
&\frac{1}{2}  \Vert \sqrt{b(t)}\Delta p(t) \Vert^2_{L^2}+\int_0^t\Vert \sqrt{r(s)} \Delta p(s) \Vert^2_{L^2} \ds\\
\lesssim &\, (1+\tau)  \Vert \sqrt{b(0)}\Delta p_0 \Vert^2_{L^2}+ \tau \|\nabla p_{2} \|_{L^2}^2+\tau^2\| \nabla p_{tt}(t) \|_{L^2}^2\\
&\,+\int_0^t \big( 1+ \tau+ \frac{1}{\tau}+\Vert b_t(s) \Vert_{L^\infty}^2\big)E[p](s) \ds+ \int_0^t \Vert f(s) \Vert_{ L^2}^2 \ds\\
\lesssim &\, (1+\tau) E[p](0) +\tau^2\| \nabla p_{tt}(t) \|_{L^2}^2+\int_0^t \big( 1+ \tau+ \frac{1}{\tau}+\Vert b_t(s) \Vert_{L^\infty}^2\big)E[p](s) \ds\\
&+ \int_0^t \Vert f(s) \Vert_{ L^2}^2 \ds.
\end{aligned}
\end{equation}
Next, we multiply the equation \eqref{Linearized_Eq_JMGT} by $-\Delta p_{tt}$ and we integrate over $\Om$ to get
\begin{equation}
\begin{aligned}
& \frac{1}{2} \ddt \big( \Vert \sqrt{\tau}\nabla p_{tt} \Vert^2_{L^2}+ \Vert \sqrt{b} \Delta p_{t} \Vert^2_{L^2} \big)+ \Vert \nabla p_{tt} \Vert^2_{L^2}\\
%=&\, -\intO (r \Delta p+b \Delta p_t) \Delta p_{tt} \dx - \intO f \Delta p_{tt} \dx\\
=&\, -\ddt \Big(\intO r \Delta p \Delta p_t \dx \Big) +\intO r_t \Delta p \Delta p_t \dx+ \intO r \vert \Delta p_t \vert^2 \dx\\
&+ \frac{1}{2} \intO b_t \vert \Delta p_t \vert^2 \dx+\intO \nabla f \cdot \nabla p_{tt} \dx 
\end{aligned}
\end{equation}
where we have used the identity  
\begin{equation} \label{f}
\begin{aligned}
%&-\intO f \Delta p_{tt} \dx = -\ddt \intO f \Delta p_t \dx +\intO %f_t \Delta p_t \dx,\\
&-\intO r \Delta p \Delta p_{tt} \dx=-\ddt \Big(\intO r \Delta p \Delta p_t \dx \Big) +\intO r_t \Delta p \Delta p_t \dx+ \intO r \vert \Delta p_t \vert^2 \dx.
\end{aligned}
\end{equation}  
Applying H\"{o}lder's inequality gives 
\begin{equation}
\begin{aligned}
& \frac{1}{2} \ddt \big( \Vert \sqrt{\tau}\nabla p_{tt} \Vert^2_{L^2}+ \Vert \sqrt{b} \Delta p_{t} \Vert^2_{L^2} \big)+ \Vert \nabla p_{tt} \Vert^2_{L^2}\\
\leq &\, \Vert r_t \Vert_{L^\infty} \Vert \Delta p \Vert_{L^2} \Vert  \Delta p_t \Vert_{L^2}+ \Vert r  \Vert_{L^\infty} \Vert \Delta p_t \Vert^2_{L^2}+ \frac{1}{2} \Vert b_t \Vert_{L^\infty} \Vert \Delta p_t \Vert^2_{L^2}\\
&+\Vert \nabla f \Vert_{L^2}  \Vert \nabla  p_{tt} \Vert_{L^2}-\ddt \Big(\intO r \Delta p \Delta p_t \dx \Big).
\end{aligned}
\end{equation}
Further, making use of Young's inequality yields
\begin{equation} 
\begin{aligned}
& \frac{1}{2} \ddt \big( \Vert \sqrt{\tau}\nabla p_{tt} \Vert^2_{L^2}+ \Vert \sqrt{b} \Delta p_{t} \Vert^2_{L^2} \big)+ \Vert \nabla p_{tt} \Vert^2_{L^2}\\
\lesssim &\, \Big\Vert \frac{1}{\sqrt{b}} \Big\Vert^2_{L^\infty}\Big(1+\Vert r  \Vert_{L^\infty}+\Vert b_t \Vert_{L^\infty}+\Vert r_t \Vert_{L^\infty}^2 \Big)  \Big(\Vert \sqrt{b} \Delta p \Vert_{L^2}^2+\Vert \sqrt{b} \Delta p_t \Vert^2_{L^2}\Big)\\
&\,  +\Vert \nabla f \Vert_{L^2}^2 +\varepsilon \Vert \nabla p_{tt} \Vert^2_{L^2}-\ddt \Big(\intO r \Delta p \Delta p_t \dx \Big).
\end{aligned}
\end{equation}
The $\varepsilon$-term on the right can be absorbed by the dissipation on the left-hand side for suitably small $\varepsilon$. Thus, recalling Assumption \ref{Assumption_1}, it follows that
\begin{equation} 
\begin{aligned}
& \frac{1}{2} \ddt \big( \Vert \sqrt{\tau}\nabla p_{tt}(t) \Vert^2_{L^2}+ \Vert \sqrt{b} \Delta p_{t}(t) \Vert^2_{L^2} \big)+ \Vert \nabla p_{tt}(t) \Vert^2_{L^2}\\
\lesssim &\,\big(1+\Vert b_t \Vert_{L^\infty}+\Vert r_t \Vert_{L^\infty}^2 \big)  E[p](t)+\Vert \nabla f (t) \Vert_{L^2}^2\\
&\, -\ddt \Big(\intO r \Delta p \Delta p_t \dx \Big).
\end{aligned}
\end{equation}
To derive abound for the last term on the right-hand side of the above estimate, we first integrate in time to get for all $0 \leq t \leq T$  
\begin{equation} \label{estimate_delta_p_t}
\begin{aligned}
& \frac{\tau}{2} \Vert \nabla p_{tt}(t) \Vert^2_{L^2}+\frac{1}{2} \Vert \sqrt{b(t)} \Delta p_{t}(t) \Vert^2_{L^2}+  \int_0^t \Vert \nabla p_{tt}(s) \Vert^2_{ L^2} \ds\\
\lesssim &\, \frac{\tau}{2} \Vert \nabla p_{2} \Vert^2_{L^2}+\frac{1}{2} \Vert \sqrt{b(0)} \Delta p_{1} \Vert^2_{L^2} +\int_0^t \big(1+ \Vert b_t(s) \Vert_{L^\infty}+ \Vert r_t (s) \Vert_{L^\infty}^2 \big)E[p](s) \ds\\
&\,+\int_0^t \Vert \nabla f(s) \Vert_{ L^2}^2 \ds +\intO \vert r(t) \Delta p(t) \Delta p_t(t) \vert \dx +\intO \vert r(0) \Delta p_0 \Delta p_1 \vert \dx.
\end{aligned}
\end{equation}
With the help of H\"{o}lder and Young inequalities together with Assumption \ref{Assumption_1}, we can show that for all $0 \leq t \leq T$, the following estimate holds 
\begin{equation} \label{ineq_1}
\begin{aligned}
&\intO \vert r(t) \Delta p(t) \Delta p_t(t) \vert \dx +\intO \vert r(0) \Delta p_0 \Delta p_1 \vert \dx\\   \leq  &\, \Big \Vert \frac{1}{\sqrt{b}} \Big \Vert_{L^\infty L^\infty}\Vert r(t) \Vert_{L^\infty} \Vert  \Delta p(t) \Vert_{L^2} \Vert \sqrt{b} \Delta p_t(t) \Vert_{L^2} \\
&\, + \Big\Vert \frac{1}{b(0)} \Big\Vert_{ L^\infty}\Vert r(0) \Vert_{L^\infty}\Vert \sqrt{b(0)} \Delta p_0 \Vert_{L^2} \Vert \sqrt{b(0)} \Delta p_1 \Vert_{L^2}\\
 \leq &\, C(\varepsilon)\Vert   \Delta p(t) \Vert_{L^2}^2+\varepsilon \Vert r \Vert_{L^\infty L^\infty}^2\Big\Vert \frac{1}{\sqrt{b}} \Big\Vert^2_{L^\infty L^\infty} \Vert \sqrt{b(t)} \Delta p_t(t) \Vert_{L^2}^2\\
&\, + \frac{1}{2}\Big\Vert \frac{1}{b} \Big\Vert_{L^\infty L^\infty}\Vert r \Vert_{L^\infty L^\infty}\big(\Vert \sqrt{b(0)} \Delta p_0 \Vert_{L^2}^2+\Vert \sqrt{b(0)} \Delta p_1 \Vert_{L^2}^2\big).
\end{aligned}
\end{equation}
By putting together the estimates \eqref{ineq_1},  \eqref{estimate_delta_p_t} and keeping in mind Assumption \ref{Assumption_1}, we choose $\varepsilon$ in \eqref{ineq_1}  suitably small so that the corresponding term will be absorbed by the left-hand side of \eqref{estimate_delta_p_t}. Thus, it results that
\begin{equation} \label{estimate_delta_p_t_1}
\begin{aligned}
& \frac{\tau}{2} \Vert \nabla p_{tt}(t) \Vert^2_{L^2}+\frac{1}{4} \Vert \sqrt{b(t)} \Delta p_{t}(t) \Vert^2_{L^2}+ \int_0^t \Vert \nabla p_{tt}(s) \Vert^2_{L^2} \ds\\
\lesssim &\, E[p](0)+\int_0^t (1+ \Vert b_t(s) \Vert_{L^\infty}+ \Vert r_t (s) \Vert_{L^\infty}^2)E[p](s) \ds+\Vert \Delta p(t) \Vert_{L^2}^2\\
&+\int_0^t \big(\Vert f(s) \Vert_{L^2 L^2}^2+ \Vert \nabla f(s) \Vert_{L^2}^2 \big) \ds.
\end{aligned}
\end{equation}
Observe that we have
\begin{equation}
\begin{aligned}
\Vert  \Delta p(t) \Vert_{L^2}^2 \lesssim&\,\Vert  \Delta p_0 \Vert_{L^2}^2+\int_0^t \Vert  \Delta p_t(s) \Vert_{L^2}^2 \ds\\
\lesssim&\,\Vert  \Delta p_0 \Vert_{L^2}^2+\Big \|\frac{1}{\sqrt{b}}\Big\|_{L^\infty}\int_0^t \Vert \sqrt{b} \Delta p_t(s) \Vert_{L^2}^2 \ds\\
\end{aligned}
\end{equation}
which implies according to Assumption \ref{Assumption_1}
\begin{equation}
\begin{aligned}
\Vert  \Delta p(t) \Vert_{L^2}^2 \lesssim&\, E[p](0)+\int_0^t E[p](s) \ds.
\end{aligned}
\end{equation}
Thus, the estimate \eqref{estimate_delta_p_t_1} becomes 
\begin{equation} \label{estimate_delta_p_t_2}
\begin{aligned}
& \frac{\tau}{2} \Vert \nabla p_{tt}(t) \Vert^2_{L^2}+\frac{1}{4} \Vert \sqrt{b(t)} \Delta p_{t}(t) \Vert^2_{L^2}+ \int_0^t \Vert \nabla p_{tt}(s) \Vert^2_{L^2} \ds\\
\lesssim &\, E[p](0)+\int_0^t \big(1+ \Vert b_t(s) \Vert_{L^\infty}+ \Vert r_t (s) \Vert_{L^\infty}^2\big)E[p](s) \ds\\
&+\int_0^t \big(\Vert f(s) \Vert_{L^2 L^2}^2+ \Vert \nabla f(s) \Vert_{L^2}^2 \big) \ds.
\end{aligned}
\end{equation}
At this point, we can add up the estimates $\lambda_0 \times$\eqref{estimate_delta_p_1} and  \eqref{estimate_delta_p_t_2} where $\lambda_0>0$ is suitably selected in order to  absorb  the term $\tau^2 \Vert \nabla p_{tt}(t) \Vert_{L^2}^2$ from the right-hand side of \eqref{estimate_delta_p_1}. This leads to the estimate 
\begin{equation} \label{interm_estimate}  
\begin{aligned}
&  E[p](0)+ \int_0^t D[p](s) \ds \\
\lesssim &\, E[p](0)+\int_0^t (1 +\Vert b_t(s) \Vert_{L^\infty}+\Vert b_t(s) \Vert_{L^\infty}^2+ \Vert r_t (s) \Vert_{L^\infty}^2)E[p](s) \ds\\
& +\int_0^t \big(\Vert f(s) \Vert_{L^2 L^2}^2+ \Vert \nabla f(s) \Vert_{L^2}^2 \big) \ds.
\end{aligned}
\end{equation}
The hidden constant in the above estimate depends on $\tau$ but not on $T$. Therefore, we arrive at the estimate \eqref{main_estimate_E_p} by applying Gronwall's inequality to \eqref{interm_estimate}. This ends the proof of Proposition \eqref{prop1}.
\end{proof}

\section{Analysis of the bioheat equation}\label{Section_Pennes_Equation}
	We focus in this section on establishing \textit{a priori} estimates for the solution of the parabolic Pennes equation:    
 	\begin{equation}\label{Heat_Linearized}
	\left\{
	\begin{aligned} 
&m \Theta_t -\kappaa \Delta \Theta+\ell \Theta = g(x,t),\qquad (x,t) \in \Om \times (0, T),\\
&\Theta\vert_{\partial \Om}=0, \\
&\Theta \vert_{t=0}=\Theta_0. 
\end{aligned}
\right. 
\end{equation}
	When formulating these estimates, our aim is   to assign the minimal regularity to the source term $g$. This is crucial for the nonlinear problem \eqref{Main_system_JMGT}, where $g$ serves as a placeholder for the source term $\Phi(\Theta, p_t)$, given the limited regularity available for the variable $p$. In addition, one can expect that $\Theta \in L^\infty(0, T; L^\infty(\Om))$ would be enough to prevent the degeneracy of the JMGT equation as was shown in the previous result \cite{Nikolic_2022}. However, here   
the sort of coupling that we have in \eqref{Main_system_JMGT}  together with the hyperbolic character of the JMGT equation and the fact that the variable coefficients are associated with the higher-order derivatives in the pressure equation suggest that we need solutions to \eqref{Heat_Linearized} with more regularity than $\Theta \in L^\infty(0, T; H^2(\Om))$. 
To accomplish this, we require the following assumptions on the initial data and the source term in \eqref{Heat_Linearized}.
\begin{assumption} \label{Assumption_2}
Assume that
\begin{itemize}
\item $g \in H^1(0, T; H^1_0(\Om))$.
\item The initial data satisfy the compatibility condition
\begin{equation}\label{theta_0} 
\Theta_0 \in H^3(\Om) \cap H^1_0(\Om), \qquad \Theta_1:= \Theta_t(x, 0) \in H^1_0(\Om),
\end{equation}
such that $\Theta_1$ is defined as 
\begin{equation} \label{theta_1} 
\Theta_1=\frac{1}{m} \big(\kappaa \Delta \Theta_0 -\ell \Theta_0+g(x, 0)\big),
\end{equation}
with 
\begin{equation}
g(0)= \Phi(\Theta_0, p_1)= \phi(\Theta_0) (p_1)^2
\end{equation}
where $\phi$ is defined as in \eqref{funct_k}.
\end{itemize}
\end{assumption} 
%Let $\mathcal{X}_\Theta^2$ denote the functional space 
%\begin{equation}
%\begin{aligned}
%\mathcal{X}_\Theta=\big\{ & \Theta \in L^2(0, T; H^3(\Om)\cap  H^1_0(\Om) ),\\
%&\, \Theta_t \in L^\infty(0, T; H^1_0(\Om)) \cap L^2(0, T; H^2(\Om)\cap H^1_0(\Om)),\\
%&\, \Theta_{tt} \in L^2(0, T; L^2(\Om)) \big\}.
%\end{aligned}
%\end{equation}
By relying again on the Faedo-Galerkin approach, we prove the next result. We only emphasize the portion of the proof leading to the energy bound \eqref{Regularity_H^2_Heat} for the smooth approximations in space of the solutions (in terms of the eigenfunctions of the Dirichlet-Laplacian) and we skip the other steps that are standard in the context of the heat equation. The reader interested in more details is referred to \cite[Chapter 7]{evans2010partial}.
\begin{proposition} \label{prop2}
Given $T>0$. Let $\Theta_0,g$ satisfy Assumption \ref{Assumption_2}. Then the problem \eqref{Heat_Linearized} admits a unique solution $\Theta \in \mathcal{X}_\Theta^2$ (cf. \eqref{spaces}). Further, there exists a constant $C_{\Theta, 2}>0$, independent of $T$, such that  
\begin{equation}\label{Regularity_H^2_Heat}
\begin{aligned}
&\| \Theta \|_{L^2 H^3}^2+\| \Theta \|_{L^\infty H^2}^2+\| \Theta_t \|_{L^\infty H^1}^2+\| \Theta_t \|_{L^2 H^2}^2+\Vert \Theta_{tt} \|_{L^2 L^2}^2\\
\leq & \,C_{\Theta, 2} \big( \| \Theta_0 \|_{H^3}^2+\| \nabla g(0) \|_{ L^2 }^2+\| g \|_{H^1 L^2}^2+\| g \|_{ L^2 H^1 }^2 \big).
\end{aligned}
\end{equation}
\end{proposition}
\begin{proof}
We multiply the equation \eqref{Heat_Linearized} by $\Theta$ and integrate in space, we obtain 
\begin{equation}
\begin{aligned}
&\frac{m}{2}\ddt \| \Theta \|_{L^{2}}^2+ \kappaa \|\nabla \Theta \|_{L^{2}}^2 +\ell \| \Theta \|_{L^{2}}^2
=\int_\Omega g \Theta \dx.
\end{aligned}
\end{equation}
We can estimate the right-hand side using H\"{o}lder and Young inequalities, to get 
\begin{equation}
\begin{aligned}
&\frac{m}{2}\ddt \| \Theta \|_{L^{2}}^2+ \kappaa \|\nabla \Theta \|_{L^{2}}^2 +\ell \| \Theta \|_{L^{2}}^2
\leq C(\varepsilon) \| g \|_{L^2}^2+ \varepsilon \|\Theta \|^2_{L^2}.
\end{aligned}
\end{equation}
The last term on the right of the inequality above can be absorbed into the left-hand side for sufficiently small values of $\varepsilon$. Thus, we obtain 
\begin{equation} \label{theta_estimate}
\begin{aligned}
&\frac{m}{2}\ddt \| \Theta \|_{L^{2}}^2+ \kappaa \|\nabla \Theta \|_{L^{2}}^2 +\ell \| \Theta \|_{L^{2}}^2
\lesssim \| g \|_{L^2}^2.
\end{aligned}
\end{equation}
Next, testing the equation \eqref{Heat_Linearized} by $-\Delta \Theta_t$ and integrating with respect to $x$, we get 
\begin{equation}
\begin{aligned}
&\frac{1}{2}\ddt \Big(\kappaa\|\Delta \Theta \|_{L^{2}}^2+ \ell\|\nabla \Theta \|_{L^{2}}^2 \Big)+m\|\nabla \Theta_t \|_{L^{2}}^2
=-\int_\Omega g \Delta\Theta_t\dx.
\end{aligned}
\end{equation}
Using H\"{o}lder and Young inequalities, we can estimate the right-hand side to obtain 
\begin{equation} \label{ineq_4}
\begin{aligned}
&\frac{1}{2}\ddt \Big(\kappaa\|\Delta \Theta \|_{L^{2}}^2+ \ell\|\nabla \Theta \|_{L^{2}}^2 \Big)+m\|\nabla \Theta_t \|_{L^{2}}^2
\leq C(\varepsilon) \| g \|_{L^2}^2+ \varepsilon \|\Delta\Theta_t \|^2_{L^2}.
\end{aligned}
\end{equation}
Further,  we differentiate  the equation \eqref{Heat_Linearized} with respect to time, we get 
\begin{equation} \label{heat_eq_t}
m \Theta_{tt} -\kappaa \Delta \Theta_t+\ell \Theta_t = g_t.
\end{equation}
We multiply equation \eqref{heat_eq_t} by $-\Delta \Theta_t$ and integrate in space to find
\begin{equation}\label{ineq_5}
\begin{aligned}
\frac{m}{2} \ddt \|\nabla \Theta_t \|_{L^{2}}^2+\kappaa \|\Delta \Theta_t \|_{L^{2}}^2+ \ell \|\nabla \Theta_t \|_{L^{2}}^2=&\,-\int_\Omega g_t \Delta\Theta_t\dx\\
\leq &\, C(\varepsilon) \| g_t \|_{L^2}^2+ \varepsilon \|\Delta\Theta_t \|^2_{L^2}.
\end{aligned}
\end{equation}
Summing up \eqref{theta_estimate}, \eqref{ineq_4} and  \eqref{ineq_5}, then  integrating the resulting inequality with respect to time, we get for suitably small values of $\varepsilon$
\begin{equation} \label{ineq_6}
\begin{aligned}
\begin{multlined}[t]
\| \Theta \|_{L^\infty L^{2}}^2+\|\nabla \Theta \|_{L^2 L^{2}}^2+\|\nabla \Theta \|_{L^\infty L^{2}}^2+\|\Delta \Theta \|_{L^\infty L^{2}}^2\\
+\|\nabla \Theta_t \|_{L^\infty L^{2}}^2+\|\nabla \Theta_t \|_{L^2 L^{2}}^2+\|\Delta \Theta_t \|_{L^2 L^{2}}^2\\
\lesssim \, \| \Theta_0 \|_{H^2}^2+\| \nabla \Theta_1 \|_{L^2}^2+\| g \|_{H^1 L^2}^2.
\end{multlined}
\end{aligned}
\end{equation}
This implies 
\begin{equation} \label{estimate_theta_H2}
\begin{aligned}
\begin{multlined}[t]
 \| \Theta \|_{L^\infty H^{2}}^2+\|\nabla \Theta \|_{L^2 L^{2}}^2+\| \Theta_t \|_{L^\infty H^1}^2+\| \Theta_t \|_{L^2 H^{2}}^2\\
\lesssim \, \| \Theta_0 \|_{H^2}^2+\| \nabla \Theta_1 \|_{L^2}^2+\| g \|_{H^1 L^2}^2.
\end{multlined}
\end{aligned}
\end{equation}
Recalling the compatibility condition \eqref{theta_1}, we have 
\begin{equation} \label{estimate_theta_1}
\| \nabla \Theta_1 \|_{L^2}^2 \lesssim \|  \Theta_0 \|_{H^3}^2 +\| \nabla g(0) \|_{L^2}^2.
\end{equation}
Note that based on the assumptions on the source term $g$, $\nabla g(x,0)$ is well-defined. Hence the estimate \eqref{estimate_theta_H2} becomes  
\begin{equation} \label{estimate_theta_H2_1}
\begin{aligned}
\begin{multlined}[t]
 \| \Theta \|_{L^\infty H^{2}}^2+\|\nabla \Theta \|_{L^2 L^{2}}^2+\| \Theta_t \|_{L^\infty H^1}^2+\| \Theta_t \|_{L^2 H^{2}}^2\\
\leq \, C_{\Theta,1} \big(\| \Theta_0 \|_{H^3}^2+\| \nabla g(0) \|_{L^2}^2+\| g \|_{H^1 L^2}^2 \big)
\end{multlined}
\end{aligned}
\end{equation} 
where the constant $C_{\Theta,1}$ does not depend on $T$.\\ 
Now, observe that we have from equation \eqref{Heat_Linearized}
\begin{equation}
\Delta \Theta = \frac{1}{\kappaa} \big( m \Theta_t +\ell \Theta -g \big).
\end{equation}
Hence,  for all $t \geq 0$, we obtain
\begin{equation} 
\begin{aligned}
\| \Delta \Theta(t) \|_{H^1}^2 & \lesssim    \|\nabla \Theta_t(t) \|_{ L^{2}}^2 +\|\nabla \Theta(t) \|_{L^{2}}^2+\| \nabla g(t) \|_{ L^2}^2.
\end{aligned}
\end{equation}
Taking  the estimate \eqref{estimate_theta_H2_1} into account and the fact that $-\Delta$ is an isomorphism from $H^3(\Om) \cap H^1_0(\Om)$ into $H^1(\Om)$, we deduce that
\begin{equation} \label{ineq_7}
\begin{aligned}  
\| \Theta \|_{L^2 H^3}^2&\, \lesssim \| \Theta_0 \|_{H^3}^2+\| \nabla g(0) \|_{L^2}^2+\| g \|_{H^1 L^2}^2+\| g \|_{L^2 H^1}^2.
\end{aligned}
\end{equation}
Further, from the equation \eqref{heat_eq_t}, we have 
\begin{equation}
\Vert \Theta_{tt} \|_{L^2 L^2}^2 \lesssim \| \Delta \Theta_t\|_{L^2 L^2}^2+ \| \Theta_t\|_{L^2 L^2}^2 +\| g_t\|_{L^2 L^2}^2,
\end{equation}
which together with \eqref{estimate_theta_H2_1} implies that
\begin{equation} \label{ineq_8}
\Vert \Theta_{tt} \|_{L^2 L^2}^2 \lesssim \| \Theta_0 \|_{H^3}^2+\| \nabla g(0) \|_{L^2}^2+\| g \|_{H^1 L^2}^2.
\end{equation}
The estimate \eqref{Regularity_H^2_Heat} follows by adding up \eqref{estimate_theta_H2_1}, \eqref{ineq_7} and \eqref{ineq_8}. This concludes the proof of the Proposition \ref{prop2}. 
\end{proof}

\section{Well-posedness of the nonlinear problem \eqref{coupled_problem_JMGT_Pennes}--Proof of Theorem \ref{local_Existence_Thoerem}}\label{Section_Well_Posdness}

Looking at the groundwork covered in the previous sections for the linearized problems \eqref{Linearized_Eq_JMGT} and  \eqref{Heat_Linearized}, we can transform the problem of showing the existence of solutions to \eqref{Main_system_JMGT} into a fixed-point problem for a suitably defined mapping, which will be solved by implementing Banach fixed-point theorem. We break down the proof of Theorem \ref{local_Existence_Thoerem} into three main steps.

\textbf{Step 1: Defining the fixed-point mapping}. Denote by $\mathbf{X}:=\mathcal{X}_p^2 \times \mathcal{X}_\Theta^2$.
 We consider the mapping $\mathcal{T}$ that associates to each $(p^\ast, \Theta^\ast) \in \mathbf{X}$ the solution $(p, \Theta) \in \mathbf{X}$ of the problem 
\begin{equation} \label{fixed_point_sys}
\left\{ 
\begin{aligned}
&\tau p_{ttt}+p_{tt}-h (\Theta^\ast)\Delta p -\zeta (\Theta^\ast) \Delta p_t =2 k (\Theta^\ast)((p^\ast_{t})^2+p^\ast p^\ast_{tt}),& \, &\text{in} &\  &\Omega \times (0,T),\\
& m \Theta_t - \kappaa \Delta \Theta + \ell \Theta = \mathcal{Q}(p^\ast_t),& \, &\text{in} &\ &\Omega \times (0,T)
\end{aligned}
\right.      
\end{equation}
supplemented with homogeneous Dirichlet boundary conditions and the initial data
\begin{equation}
\begin{aligned}
(p_0, p_1, p_2) \in&\, \big[H^2(\Om) \cap H^1_0(\Om) \big] \times  \big[H^2(\Om) \cap H^1_0(\Om) \big] \times H^1_0(\Om),\\
\Theta_0 \in&\, \big[H^3(\Om) \cap H^1_0(\Om) \big],
\end{aligned} 
\end{equation}
 i.e., $\mathcal{T}(p^\ast, \Theta^\ast)=(p, \Theta).$  \\
Since we intend to apply Banach fixed-point theorem, we are particularly interested in solutions of \eqref{fixed_point_sys} corresponding to $(p^\ast, \Theta^\ast) \in B \subset \mathbf{X}$ where the subset $B$ is given as
\begin{equation}
\begin{aligned}
B=\Big\{& (p^\ast, \Theta^\ast) \in  \mathbf{X}:  (p^\ast(0), p^\ast_t(0),p_{tt}^\ast(0), \Theta^\ast(0))=(p_0, p_1, p_2,  \Theta_0), \\
&  \qquad \qquad \| p^\ast \|_{\mathcal{X}_p^1} \leq \eta, \quad  \| p^\ast \|_{\mathcal{X}_p^2} \leq R_1,  \quad \Vert \Theta^\ast  \Vert_{\mathcal{X}_\Theta^2} \leq R_2 \Big\}.    
\end{aligned}  
\end{equation}
The product space $\mathbf{X}$ is endowed with the norm 
$$\Vert(p, \Theta) \Vert_{\mathbf{X}}^2=\Vert p \Vert_{\mathcal{X}_p^2}^2+\Vert \Theta \Vert_{\mathcal{X}_\Theta}^2$$
where  $\| \cdot \|_{\mathcal{X}_p^1}, \| \cdot \|_{\mathcal{X}_p^2}$ and $ \| \cdot \|_{\mathcal{X}_\Theta^2}$ denote, respectively the norms on the solutions spaces $ \mathcal{X}_p^1, \mathcal{X}_p^2, \mathcal{X}_\Theta^2$ (cf. \eqref{spaces}, \eqref{space_p_1}); that is  
\begin{equation} \label{norms}
\begin{aligned}
\Vert p \Vert_{\mathcal{X}_p^1}^2:=&\,\Vert p\Vert_{L^\infty H^1}^2+ \Vert p_t\Vert_{L^\infty H^1}^2+\tau \Vert  p_{tt} \Vert_{L^\infty L^2}^2,\\
\Vert p \Vert_{\mathcal{X}_p^2}^2:=&\,\Vert p\Vert_{L^\infty H^2}^2+ \Vert p_t\Vert_{L^\infty H^2}^2+\tau \Vert  p_{tt} \Vert_{L^\infty H^1}^2 ,\\
\Vert \Theta \Vert_{\mathcal{X}_\Theta^2}^2:=&\,\Vert \Theta \Vert_{L^2 H^3}^2+\Vert  \Theta_t \Vert_{L^\infty H^1}^2+ \| \Theta_t \|_{L^2 H^2}^2.
%+\Vert \Theta_{tt}\Vert_{L^2 L^2}^2.
\end{aligned}
\end{equation} 
We can check that the set $B$ is nonempty by considering the solution to the linear problems \eqref{Linearized_Eq_JMGT}, \eqref{Heat_Linearized} in the particular case where $r=b=1$ and $f=g=0$. Then the solution $(p, \Theta) \in \mathbf{X}$ will belong to the ball $B$ provided that we select $R_1,R_2, \eta>0$ such that
\begin{equation}
\begin{aligned}
&C_{p,1} \exp(T) \big(E_0[p](0)+E_1[p](0)\big) \leq C_{p,1} \exp(T) \eta_0^2 \leq \eta^2,\\
&C_{p,2} \exp(T) E[p](0) \leq R_1^2, \qquad C_{\Theta,2} \Vert \Theta_0 \Vert_{H^3}^2 \leq R_2^2
\end{aligned}
\end{equation}
where $C_{p,1}, C_{p,2}, C_{\Theta,2}$ are the constants appearing  respectively in the estimates \eqref{Main_First_Order_Est},  \eqref{main_estimate_E_p}, and \eqref{Regularity_H^2_Heat}. Based on the findings in Propositions \ref{prop1} and  \ref{prop2}, we can show that the mapping $\mathcal{T}$ is indeed well-defined. This will be done next along with proving that $\mathcal{T}$ is a self-mapping.

\textbf{Step 2: Showing the self-mapping property}.
Motivated by previous results in the literature  \cite{bongarti2021vanishing, Kaltenbacher_Nikolic_2022, kaltenbacher2012well}, we can show that $\mathcal{T}$ is a self-mapping under a smallness condition on the lower norm in $\mathcal{X}_p^1$. Since the existence of a fixed-point of $\mathcal{T}$ means the existence of a solution to the nonlinear problem \eqref{Main_system_JMGT} that lies in $B$, this smallness constraint also serves to guarantee that the JMGT equation in \eqref{Main_system_JMGT} does not degenerate. Specifically, we need $1-2k(\Theta) p$ to be positive at all times, which translates into choosing $\eta$ small enough so as $1-2k_1\eta^{1-\frac{d}{4}} R^{\frac{d}{4}}>0$. Indeed, this is enough to avoid degeneracy. More precisely,   using \eqref{k_1} and  inequality \eqref{Agmon_Inequality} give
\begin{equation}
\begin{aligned}
2 \| k(\Theta) p \|_{L^\infty L^\infty} \leq &\,2  \| k(\Theta) \|_{L^\infty L^\infty} \| p \|_{L^\infty L^\infty}\\
\leq &\, 2 k_1 \| p \|_{L^\infty L^2}^{1-\frac{d}{4}} \| p \|_{L^\infty H^2}^{\frac{d}{4}}\\
\leq &\, 2 k_1 \eta^{1-\frac{d}{4}} R_1^{\frac{d}{4}}.
\end{aligned}
\end{equation}
At this stage, no smallness conditions are imposed on $T$ and $R_1$.  
Our approach is similar to that taken in  \cite{bongarti2021vanishing, Nikolic_2022}.
\begin{lemma}\label{Lemma_Self_Mapping}
Given $\tau >0$. Let $T>0$ be arbitrary. There exist   $0<\eta_0\ll  \eta<1$ are sufficiently small, and $R_1$ and $R_2$  such that if the initial data satisfy
$$E_0[p](0)+E_1[p](0) \leq \eta_0^2$$
and 
\begin{equation}
C_{p,1}  (\eta_0^2 +C_f T (\eta^4+ \eta^{4-\frac{d}{2}} R_1^{\frac{d}{2}}))\exp\big((1+\sqrt{T}+T)C_1( R_2)\big) \leq \eta^2,
\end{equation}
then we have $\mathcal{T}(B) \subset B$. The constants $ C_{p,1}, C_1, C_f$ are independent of $T$. 
\end{lemma}
\begin{proof}
We aim to prove that for $(p^\ast, \Theta^\ast) \in B$, the solution to  the linear initial boundary value problem \eqref{fixed_point_sys} defined as  $(p, \Theta)=\mathcal{T}(p^\ast, \Theta^\ast)$ exists in $B$. To this end, we exploit what we have already established in Propositions \ref{lemma_lower_norms},   \ref{prop1} and  \ref{prop2}. 
 To apply the linear analysis, we need first to show that the $p^\ast$ and $\Theta^\ast$ terms in \eqref{fixed_point_sys} satisfy all the hypotheses of the linear existence-uniqueness result. For this purpose, 
we begin by setting
\begin{equation}\label{the function_f}
\begin{aligned}
    & r(x,t)=h (\Theta^\ast), & \qquad & b(x,t)=\zeta (\Theta^\ast),\\
    &f(x,t)=2k (\Theta^\ast)\big((p^\ast_t)^2+p^\ast p_{tt}^\ast \big),& \qquad & g(x,t)=\Phi( \Theta^\ast, p^\ast_t)=\phi(\Theta^\ast) (p_t^\ast)^2.
    \end{aligned}
\end{equation}
We know from Assumption \ref{Assumption_nonlinear} that there exist $h_1, \zeta_1 >0$ such that
$$ h_1 \leq h (\Theta^\ast)=r(x,t) \quad \text{and}\quad  \zeta_1 \leq \zeta (\Theta^\ast)=b(x,t).$$
For convenience, we recall the definitions of the functions $\Lambda_0, \Lambda$ appearing, respectively  in Proposition  \ref{lemma_lower_norms}, and  Proposition \ref{prop1}
\begin{equation}
\begin{aligned}
&\Lambda_0(t)=1+ \| r_t(t)\|_{L^\infty}+ \| r_t(t) \|_{L^\infty}^2+\| b_t(t) \|_{L^\infty} +\| \nabla b(t) \|^2_{L^\infty}+\| \nabla r(t) \|^2_{L^\infty},\\
&\Lambda(t)= 1 + \Vert b_t (t) \Vert_{L^\infty}+\Vert b_t(t) \Vert_{L^\infty}^2+ \Vert r_t(t) \Vert_{L^\infty}^2, \quad 0 \leq t \leq T.
\end{aligned}
\end{equation}
According to the properties of the function $\zeta$ (see \eqref{zeta}) and the embedding $H^2(\Om) \hookrightarrow L^\infty(\Om)$, we can see that
\begin{equation} \label{ineq_gradient_b}
\begin{aligned}
\Vert \nabla b \Vert_{L^2 L^\infty}^2= \Vert \zeta' (\Theta^\ast) \nabla \Theta^\ast \Vert_{L^2 L^\infty}^2\leq &\, \Vert \zeta' (\Theta^\ast) \Vert_{L^\infty L^\infty}^2 \Vert  \nabla \Theta^\ast \Vert_{L^2 L^\infty}^2\\
\lesssim &\, (1+ \Vert  \Theta^\ast \Vert_{L^\infty L^\infty}^{\gamma_3+1})^2 \Vert \nabla \Theta^\ast \Vert_{L^2 L^\infty}^2\\
\lesssim &\, (1+ \Vert  \Theta^\ast \Vert_{L^\infty L^\infty}^{2\gamma_3+2}) \Vert \nabla \Theta^\ast \Vert_{L^2 H^2}^2\\
\lesssim &\, (1+ \Vert  \Theta^\ast \Vert_{L^\infty H^2}^{2\gamma_3+2}) \Vert \Theta^\ast \Vert_{L^2 H^3}^2\\
\lesssim &\, (1+ R_2^{2\gamma_3+2})R_2^2.
\end{aligned}
\end{equation}
Analogously, we have the estimate 
\begin{equation} \label{ineq_b_t}
\begin{aligned}
\Vert b_t \Vert_{L^2 L^\infty}^2 =  \Vert \zeta' (\Theta^\ast) \Theta^\ast_t \Vert_{L^2 L^\infty}^2\leq &\, \Vert \zeta' (\Theta^\ast) \Vert_{L^\infty L^\infty}^2 \Vert  \Theta^\ast_t \Vert_{L^2 L^\infty}^2 \\
\lesssim &\, (1+ \Vert  \Theta^\ast \Vert_{L^\infty L^\infty}^{\gamma_3+1})^2 \Vert \Theta^\ast_t \Vert_{L^2 L^\infty}^2\\
\lesssim &\, (1+ \Vert  \Theta^\ast \Vert_{L^\infty H^2}^{2\gamma_3+2}) \Vert \Theta^\ast_t \Vert_{L^2 H^2}^2\\
\lesssim &\, (1+ R_2^{2\gamma_3+2})R_2^2.
\end{aligned}
\end{equation}
Thus on account of the embedding $L^2 (0, T) \hookrightarrow L^1(0, T)$ and the fact that $(p^\ast, \Theta^\ast) \in B$, we get  
\begin{equation} \label{estimates_b} 
\Vert \nabla b \Vert_{L^2 L^\infty}^2+\Vert b_t \Vert_{L^2 L^\infty}^2+\Vert b_t \Vert_{L^1 L^\infty} \leq (1+\sqrt{T})C(R_2).
\end{equation}
Similarly, we can obtain estimates for $r_t$ and $\nabla r$ by using the polynomial growth of the function $h$ \eqref{h'_assump}. Indeed, we have thanks to the embedding $H^2(\Om) \hookrightarrow L^\infty(\Om)$
\begin{equation}
\begin{aligned}
\Vert r_t \Vert_{L^2 L^\infty}^2=\Vert h '(\Theta^\ast) \Theta^\ast_t\Vert_{L^2 L^\infty}^2 &\, \lesssim (1+\Vert  \Theta^\ast \Vert_{L^\infty L^\infty}^{\gamma_1+1})^2\Vert \Theta^\ast_t\Vert_{L^2 L^\infty}^2\\
    &\, \lesssim (1+\Vert  \Theta^\ast \Vert_{L^\infty H^2}^{2\gamma_1+2})\Vert \Theta^\ast_t\Vert_{L^2 H^2}^2\\
   &\, \lesssim   (1+ R_2^{2\gamma_1+2})R_2^2,
\end{aligned}
\end{equation}
and
\begin{equation}
\begin{aligned}
\Vert \nabla r \Vert_{L^2 L^\infty}^2=\Vert h '(\Theta^\ast) \nabla \Theta^\ast \Vert_{L^2 L^\infty}^2 &\, \lesssim (1+\Vert  \Theta^\ast \Vert_{L^\infty L^\infty}^{\gamma_1+1})^2\Vert \nabla \Theta^\ast\Vert_{L^2 L^\infty}^2\\
    &\, \lesssim (1+\Vert  \Theta^\ast \Vert_{L^\infty H^2}^{2\gamma_1+2})\Vert \Theta^\ast \Vert_{L^2 H^3}^2\\
    &\, \lesssim   (1+ R_2^{2\gamma_1+2})R_2^2.
\end{aligned}
\end{equation}
The above estimates along with the embedding $L^2(0, T) \hookrightarrow L^1(0, T)$ and the fact that $(p^\ast, \Theta^\ast) \in B$ yield 
\begin{equation} \label{estimates_r} 
\Vert \nabla r \Vert_{L^2 L^\infty}+\Vert r_t \Vert_{L^2 L^\infty}+\Vert r_t \Vert_{L^1 L^\infty} \leq (1+\sqrt{T})C(R_2).
\end{equation} 
Hence, by putting together the estimates \eqref{estimates_b}, \eqref{estimates_r}, we infer that
\begin{equation} 
\Vert \Lambda \Vert_{L^1(0,T)}+ \Vert \Lambda_0 \Vert_{L^1(0,T)} \leq (1+\sqrt{T}+T) C_1(R_2)
\end{equation}
where here and below $C_i(R_2),\, i=1,2,3$ denotes a constant depending polynomially on $R_2$ but not on $T$.

Now we focus on the source term $f$. Keeping in mind \eqref{the function_f}, we have
\begin{equation} \label{estimate_f}  
\begin{aligned}
    \Vert f \Vert_{L^2 L^2}= &\Vert 2k (\Theta^\ast) (( p^\ast_t)^2 + p^\ast p^\ast_{tt}) \Vert_{L^2 L^2} \\
    \leq &\,  \Vert  k (\Theta^\ast ) \Vert_{L^\infty L^\infty} ( \Vert p^\ast_t  \Vert_{L^\infty L^4} \Vert p^\ast_t  \Vert_{L^2 L^4} + \Vert p^\ast \Vert_{L^\infty L^\infty} \Vert p^\ast_{tt} \Vert_{L^2 L^2})\\
    \leq &\, C \sqrt{T} k_1  ( \Vert p^\ast_t  \Vert_{L^\infty H^1}^2+ \|p^\ast \|_{L^\infty L^2}^{1-\frac{d}{4}} \|p^\ast \|_{L^\infty H^2}^{\frac{d}{4}}\Vert p^\ast_{tt} \Vert_{L^\infty L^2})\\
    \lesssim &\,  \sqrt{T}( \| p^\ast \|_{\mathcal{X}_p^1}^2+\| p^\ast \|_{\mathcal{X}_p^1}^{2-\frac{d}{4}}\| p^\ast \|_{\mathcal{X}_p^2}^{\frac{d}{4}})\\
    \lesssim &\, \sqrt{T} (\eta^2+ \eta^{2-\frac{d}{4}} R_1^{\frac{d}{4}}),
    \end{aligned}
\end{equation}
where we have used the properties of the function $k$ (see \eqref{k_1}), the embeddings $H^1(\Om) \hookrightarrow L^4(\Om)$, $L^\infty(0, T) \hookrightarrow L^2(0, T)$, Agmon's inequality \eqref{Agmon_Inequality} and the fact that $(p^\ast, \Theta^\ast) \in B$.\\
In the same manner, we can get a bound for $\Vert \nabla f \Vert_{L^2 L^2}$ as follows: 
\begin{equation} \label{gradient_f_est_0}
\begin{aligned}
\begin{multlined}[t]
     \Vert \nabla f \Vert_{L^2 L^2}    \leq \,\Vert 2 k (\Theta^\ast) (2 p^\ast_t \nabla p^\ast_{t}+ \nabla p^\ast  p^\ast_{tt}+ p^\ast \nabla p^\ast_{tt})\Vert_{L^2 L^2}\\
      +\Vert 2k '(\Theta^\ast) \nabla \Theta^\ast ((p^\ast_t)^2 + p^\ast p^\ast_{tt})\Vert_{L^2 L^2}\\
    \lesssim \, \| k (\Theta^\ast) \|_{L^\infty L^\infty} \Big( \| p^\ast_t \|_{L^\infty L^\infty} \| \nabla p^\ast_{t}\|_{L^2 L^2}+\|\nabla p^\ast \|_{L^\infty L^4} \| p^\ast_{tt} \|_{L^2 L^4}+  \|p^\ast \|_{L^\infty L^\infty} \| \nabla p^\ast_{tt} \|_{L^2 L^2}\Big)\\
    +\Vert  k'(\Theta^\ast ) \Vert_{L^\infty L^\infty} \Vert  \nabla \Theta^\ast  \Vert_{L^\infty L^4} \Big(\Vert p^\ast_t \Vert_{L^2 L^4} \Vert p^\ast_t \Vert_{L^\infty L^\infty}+ \|p^\ast \|_{L^\infty L^\infty} \| p^\ast_{tt} \|_{L^2 L^4}\Big).
    \end{multlined}
    \end{aligned}
\end{equation}
Recalling the properties of the function $k$ \eqref{k_1}, \eqref{properties_k} and using the embeddings $H^2(\Om) \hookrightarrow L^\infty(\Om)$, $H^1(\Om) \hookrightarrow L^4(\Om)$,  $L^\infty(0, T) \hookrightarrow L^2(0, T)$, together with the  interpolation inequalities \eqref{Agmon_Inequality}, \eqref{Lady_Inequality}, it follows
\begin{equation}  \label{gradient_f_est}
\begin{aligned}
\begin{multlined}[t]
    \Vert \nabla f \Vert_{L^2 L^2} \lesssim \, \sqrt{T} \Big[ \|p^\ast_t \|_{L^\infty L^2}^{1-\frac{d}{4}} \|p^\ast_t \|_{L^\infty H^2}^{\frac{d}{4}} \| \nabla p^\ast_{t}\|_{L^\infty L^2}+  \|\nabla p^\ast \|_{L^\infty L^2}^{1-\frac{d}{4}} \|\nabla p^\ast \|_{L^\infty H^1}^{\frac{d}{4}} \| \nabla p^\ast_{tt} \|_{L^\infty L^2}\\
    +  \|p^\ast \|_{L^\infty L^2}^{1-\frac{d}{4}} \|p^\ast \|_{L^\infty H^2}^{\frac{d}{4}} \| \nabla p^\ast_{tt} \|_{L^\infty L^2}\\
      + (1+\Vert \Theta^\ast \Vert_{L^\infty L^\infty}^{\gamma_2+1}) \Vert \nabla \Theta^\ast  \Vert_{L^\infty H^1} \Big( \Vert p^\ast_t \Vert_{L^\infty H^1} \|p^\ast_t \|_{L^\infty L^2}^{1-\frac{d}{4}} \|p^\ast_t \|_{L^\infty H^2}^{\frac{d}{4}}\\
     + \|p^\ast \|_{L^\infty L^2}^{1-\frac{d}{4}} \|p^\ast \|_{L^\infty H^2}^{\frac{d}{4}} \| p^\ast_{tt} \|_{L^\infty H^1} \Big) \Big]\\
     \lesssim \, \sqrt{T}\Big(1+\| \Theta^\ast \|_{\mathcal{X}_\Theta^2}+\| \Theta^\ast \|_{\mathcal{X}_\Theta^2}^{\gamma_2+2}\Big)\Big(\| p^\ast \|_{\mathcal{X}_p^1}^{2-\frac{d}{4}}\| p^\ast \|_{\mathcal{X}_p^2}^{\frac{d}{4}}+ \| p^\ast \|_{\mathcal{X}_p^1}^{1-\frac{d}{4}}\| p^\ast \|_{\mathcal{X}_p^2}^{1+\frac{d}{4}} \Big).
     \end{multlined}
    \end{aligned}
\end{equation}
Since $(p^\ast, \Theta^\ast) \in B$, we get
\begin{equation} \label{estimate_f_t} 
\Vert \nabla f \Vert_{L^2 L^2} \lesssim \sqrt{T}(1+R_2+R_2^{2+\gamma_2})(\eta^{2-\frac{d}{4}} R_1^{\frac{d}{4}}+\eta^{1-\frac{d}{4}} R_1^{1+\frac{d}{4}}).
\end{equation}
Hence,  combining \eqref{estimate_f} and \eqref{estimate_f_t}, we obtain the following estimate for the function $F$ in Proposition \ref{prop1} 
\begin{equation} \label{estimate_F}
\|F\|_{L^1(0, T)} \leq C_f T \Big((\eta^4+ \eta^{4-\frac{d}{2}} R_1^{\frac{d}{2}})+ R_1^{\frac{d}{2}}(\eta^{4-\frac{d}{2}} +\eta^{2-\frac{d}{2}} R_1^{2})C_2(R_2) \Big)
\end{equation}
where the constant $C_f$ does not depend on $T$.\\
Therefore, all the hypotheses on the functions $r,\, b$ and $f$ in Proposition \eqref{prop1} are fulfilled. Thus,  we can deduce the existence of a unique solution $p \in \mathcal{X}_p^2$ to the first equation in system \eqref{fixed_point_sys}. In addition, since $( \Theta^\ast, p^\ast) \in B$,  we have $g= \Phi( \Theta^\ast, p^\ast_t) \in H^1(0, T ; H^1_0(\Om))$, which ensures, according to Proposition \ref{prop2}, the existence of a unique solution $\Theta \in \mathcal{X}_\Theta^2$ to the second equation in \eqref{fixed_point_sys}. We can then conclude  that the mapping $\mathcal{T}$ is well-defined. 
Furthermore, the functions $p$ and $\Theta$ satisfy, respectively  the estimates \eqref{main_estimate_E_p} and \eqref{Regularity_H^2_Heat} provided in Propositions \ref{prop1} and  \ref{prop2}. To take advantage of these bounds, we first need upper bounds for the source term $g$. To this end, we apply H\"{o}lder's inequality to get the   
\begin{equation} \label{g(0)_estimate}
\begin{aligned}
\Vert \nabla g(0) \Vert_{L^2}& =\,\Vert \phi'(\Theta_0) \nabla \Theta_0(p_1)^2 +2\phi(\Theta_0)p_1 \nabla p_1 \Vert_{L^2}\\
&
\begin{multlined}[t]
\lesssim \, \Vert \phi'(\Theta_0) \nabla \Theta_0(p_1)^2  \Vert_{L^2}+\Vert \phi(\Theta_0)p_1 \nabla p_1 \Vert_{L^2}\\
\lesssim \, \Vert \phi'(\Theta_0)  \Vert_{L^\infty} \Vert \nabla \Theta_0\Vert_{L^\infty} \Vert p_1  \Vert_{L^4}^2+\Vert \phi(\Theta_0) \Vert_{L^\infty} \Vert p_1 \Vert_{L^4} \Vert \nabla p_1 \Vert_{L^4}\\
\lesssim \,(1+ \| \Theta_0 \|_{H^3}^{1+\gamma_4}) \| \Theta_0 \|_{H^3} \| p_1 \|_{H^2}^2+ \phi_1 \| p_1 \|_{H^2}^2\\
\leq \, C_g(1+ \| \Theta_0 \|_{H^3} +\| \Theta_0 \|_{H^3}^{2+\gamma_4}) \| p_1 \|_{H^2}^2.
\end{multlined}
\end{aligned}
\end{equation}
We arrived at the last bound in \eqref{g(0)_estimate} by exploiting the properties of the function $\phi$ (see \eqref{phi_1} and  \eqref{phi_inequalities}) along with the embeddings $H^2(\Om) \hookrightarrow L^\infty(\Om)$, $H^1(\Om) \hookrightarrow L^4(\Om)$. The constant $C_g>0$ may depend on the medium parameters but not on $T$. \\
Moreover, recalling \eqref{the function_f}, we can rely on the same tools to obtain
\begin{equation} 
\begin{aligned}
   \Vert \nabla g \Vert_{L^2 L^2}^2+ \Vert g_t \Vert_{L^2 L^2}^2
 =&\, \Vert \phi'(\Theta^\ast) \nabla \Theta^\ast (p^\ast_t)^2 +2\phi(\Theta^\ast)p^\ast_t \nabla p^\ast_t \Vert_{L^2 L^2}^2\\
 &+ \Vert \phi'(\Theta^\ast)  \Theta_t^\ast (p_t^\ast)^2 +2\phi(\Theta^\ast)p_t^\ast  p_{tt}^\ast \Vert_{L^2 L^2}^2\\
 \lesssim& \, \Vert \phi'(\Theta^\ast) \nabla \Theta^\ast (p^\ast_t)^2 \Vert_{L^2 L^2}^2 +\Vert \phi(\Theta^\ast)p^\ast_t \nabla p^\ast_t \Vert_{L^2 L^2}^2\\
 &
 \begin{multlined}[t]
 + \Vert \phi'(\Theta^\ast)  \Theta_t^\ast (p_t^\ast)^2 \Vert_{L^2 L^2}^2 
 +\Vert  \phi(\Theta^\ast)p_t^\ast  p_{tt}^\ast \Vert_{L^2 L^2}^2. 
 \end{multlined}
 \end{aligned}
\end{equation}
Hence, this gives 
\begin{equation} \label{g_estimate_0}
\begin{aligned}
 \Vert \nabla g \Vert_{L^2 L^2}^2+ \Vert g_t \Vert_{L^2 L^2}^2\lesssim &\, \Vert \phi'(\Theta^\ast)  \Vert_{L^\infty L^\infty}^2 \Vert \nabla \Theta^\ast \Vert_{L^\infty L^4}^2  \Vert p^\ast_t \Vert_{L^\infty L^\infty}^2 \Vert p^\ast_t \Vert_{L^2 L^4}^2 \\
 &
 \begin{multlined}[t]
 +\Vert \phi(\Theta^\ast) \Vert_{L^\infty L^\infty}^2 \Vert p^\ast_t \Vert_{L^\infty L^\infty}^2 \Vert \nabla p^\ast_t \Vert_{L^2 L^2}^2\\
 + \Vert \phi'(\Theta^\ast) \Vert_{L^\infty L^\infty}^2  \Vert \Theta_t^\ast \Vert_{L^\infty L^4}^2 \Vert p^\ast_t \Vert_{L^\infty L^\infty}^2 \Vert p^\ast_t \Vert_{L^2 L^4}^2\\
 +\Vert  \phi(\Theta^\ast) \Vert_{L^\infty L^\infty}^2  \Vert p_t^\ast  \Vert_{L^\infty L^\infty}^2  \Vert p_{tt}^\ast \Vert_{L^2 L^2}^2
 \end{multlined}
\end{aligned}
\end{equation}
Consequently, we obtain 
\begin{equation}\label{g_estimate}
\begin{aligned}
\Vert \nabla g \Vert_{L^2 L^2}^2+ \Vert g_t \Vert_{L^2 L^2}^2
 \lesssim &\, T(1+ \Vert \Theta^\ast  \Vert_{L^\infty L^\infty}^{1+\gamma_4})^2 \Vert \Theta^\ast \Vert_{L^\infty H^2}^2  \Vert p^\ast_t \Vert_{L^\infty L^\infty}^2 \Vert p^\ast_t \Vert_{L^\infty H^1}^2 \\
 &\begin{multlined}
 + T \phi_1^2 \Vert p^\ast_t \Vert_{L^\infty L^\infty}^2 \Vert \nabla p^\ast_t \Vert_{L^\infty L^2}^2\\
 + T(1+ \Vert \Theta^\ast  \Vert_{L^\infty L^\infty}^{1+\gamma_4})^2 \Vert \Theta_t^\ast \Vert_{L^\infty H^1}^2 \Vert p^\ast_t \Vert_{L^\infty L^\infty}^2 \Vert p^\ast_t \Vert_{L^\infty H^1}^2\\
 +T \phi_1^2 \Vert p_t^\ast  \Vert_{L^\infty L^\infty}^2  \Vert p_{tt}^\ast \Vert_{L^\infty L^2}^2.
 \end{multlined}
\end{aligned}
\end{equation}
We can get more from the estimate \eqref{g_estimate} by using Agmon's inequality \eqref{Agmon_Inequality}. In fact, we have
\begin{equation} \label{g_estimate_1}
\begin{aligned}
 & \Vert \nabla g \Vert_{L^2 L^2}^2+ \Vert g_t \Vert_{L^2 L^2}^2\\
 \lesssim &\, T(1+ \Vert \Theta^\ast  \Vert_{L^\infty L^\infty}^{2+2\gamma_4}) \Vert \Theta^\ast \Vert_{L^\infty H^2}^2  \Vert p_t^\ast \Vert_{L^\infty L^2}^{2-\frac{d}{2}} \Vert p_t^\ast \Vert_{L^\infty H^2}^{\frac{d}{2}} \Vert p^\ast_t \Vert_{L^\infty H^1}^2 \\
 &\begin{multlined}[t]
 + T \Vert p_t^\ast \Vert_{L^\infty L^2}^{2-\frac{d}{2}} \Vert p_t^\ast \Vert_{L^\infty H^2}^{\frac{d}{2}} \Vert \nabla p^\ast_t \Vert_{L^\infty L^2}^2\\
 + T(1+ \Vert \Theta^\ast  \Vert_{L^\infty L^\infty}^{2+2\gamma_4}) \Vert \Theta_t^\ast \Vert_{L^\infty H^1}^2 \Vert p_t^\ast \Vert_{L^\infty L^2}^{2-\frac{d}{2}} \Vert p_t^\ast \Vert_{L^\infty H^2}^{\frac{d}{2}} \Vert p^\ast_t \Vert_{L^\infty H^1}^2\\
 +T \Vert p_t^\ast \Vert_{L^\infty L^2}^{2-\frac{d}{2}} \Vert p_t^\ast \Vert_{L^\infty H^2}^{\frac{d}{2}}  \Vert p_{tt}^\ast \Vert_{L^\infty L^2}^2
 \end{multlined}\\
\lesssim &\, T(1+ \| \Theta^\ast \|_{\mathcal{X}_\Theta^2}^2+ \| \Theta^\ast \|_{\mathcal{X}_\Theta^2}^{4+ 2 \gamma_4}) \Vert p^\ast \Vert_{\mathcal{X}_p^1}^{4-\frac{d}{2}} \Vert p^\ast \Vert_{\mathcal{X}_p^2}^{\frac{d}{2}}\\
\lesssim &\, T(1+ R_2^2+ R_2^{4+ 2 \gamma_4}) \eta^{4-\frac{d}{2}}R_1^{\frac{d}{2}}\\
\leq &\, T \eta^{4-\frac{d}{2}}R_1^{\frac{d}{2}} C_3(R_2).
\end{aligned}
\end{equation}
Since $p_t^\ast|_{\partial \Om}=0$, the estimate \eqref{Regularity_H^2_Heat} together with \eqref{g(0)_estimate} and  \eqref{g_estimate_1} lead  to 
\begin{equation} \label{theta_estimate_self_mapping} 
\begin{aligned}
\| \Theta \|_{\mathcal{X}_\Theta^2}^2 \leq &\, C_{\Theta, 2} \Big(\Vert \Theta_0 \Vert_{H^3}^2 + \Vert \nabla g(0) \Vert_{L^2}^2+\Vert \nabla g \Vert_{L^2 L^2}^2+\Vert  g_t \Vert_{L^2 L^2}^2 \Big)\\
\leq &\,C_{\Theta, 2}\Big( \Vert \Theta_0 \Vert_{H^3}^2+ C_g(1+ \| \Theta_0 \|_{H^3}^2+\| \Theta_0 \|_{H^3}^{4+2\gamma_4})\| p_1 \|_{H^2}^4\\
&\qquad\qquad  +  T \eta^{4-\frac{d}{2}}R_1^{\frac{d}{2}}C_3(R_2) \Big)
\end{aligned}
\end{equation}
As for the solution to the pressure equation in \eqref{fixed_point_sys} $p \in \mathcal{X}_p^2$, we know from Proposition \ref{prop1} that $p$ satisfies the estimate \eqref{main_estimate_E_p}, this when  combined with \eqref{estimate_F} results in 
\begin{equation} \label{p_estimate_self_mapping}
\begin{aligned}
 \Vert p \Vert_{\mathcal{X}_p^2}^2 \leq &\, \sup_{t \in (0,T)} E[p](t)+ \int_0^T D[p](s) \ds\\
\leq &\, C_{p,2} \Big(E[p](0)+C_f T \Big(\eta^4+ \eta^{4-\frac{d}{2}} R_1^{\frac{d}{2}}+ R_1^{\frac{d}{2}}(\eta^{4-\frac{d}{2}} +\eta^{2-\frac{d}{2}} R_1^{2})C_2(R_2) \Big)\Big)\\
&\times \exp\big((1+\sqrt{T}+T) C_1( R_2)\big).
\end{aligned}
\end{equation}
Let us denote
\begin{equation}
R_0^2:= E[p](0)+\Vert \Theta_0 \Vert_{H^3}^2+ C_g(1+ \| \Theta_0 \|_{H^3}^2+\| \Theta_0 \|_{H^3}^{4+2\gamma_4})\| p_1 \|_{H^2}^4.
\end{equation}
Therefore, in order to get the bounds 
\begin{equation} \label{p_theta_R}
\Vert p \Vert_{\mathcal{X}_p^2}^2 \leq R_1^2, \qquad \| \Theta \|_{\mathcal{X}_\Theta^2}^2 \leq R_2^2
\end{equation}
we first choose $R_2$ such that
%provided that $R_1>0$ is taken sufficiently large such that 
\begin{equation}
\begin{aligned}
&  2 C_{\Theta, 2} R_0 ^2\leq  R_2^2
\end{aligned}
\end{equation}
Once $R_2$ is fixed, we select $R_1$ large enough such that 
\begin{equation}
 2C_{p,2} R_0 ^2\exp\big((1+\sqrt{T}+T)C_1(R_2)\big) \leq R_1^2. 
\end{equation}
After that,  we take $\eta>0$ sufficiently small so as to satisfy
\begin{equation}
2 C_{\Theta, 2} T \eta^{4-\frac{d}{2}}R_1^{\frac{d}{2}}C_3(R_2) \leq R^2_2 
\end{equation}
and 
\begin{equation}
\begin{aligned}
2 C_{p,2} T  C_f\Big(\eta^4+ \eta^{4-\frac{d}{2}} R_1^{\frac{d}{2}}+ R_1^{\frac{d}{2}}(\eta^{4-\frac{d}{2}} +\eta^{2-\frac{d}{2}} R_1^{2})C_2(R_2) \Big)\\
\, \times \exp\big((1+\sqrt{T}+T) C_1(R_2)\big) \leq R_1^2.
\end{aligned}
\end{equation}
On the other hand, we know that the solution $p$ satisfies the estimate \eqref{Main_First_Order_Est} 
\begin{equation}
\begin{aligned}
\Vert p \Vert_{\mathcal{X}_p^1}^2  \leq &\, \sup_{t \in (0,T)} \{E_0[p](t)+E_1[p](t)\}+ \int_0^T D_0[p](s) \ds\\
\leq &\, C_{p,1}  \Big(E_0[p](0)+E_1[p](0)+T (\eta^4+ \eta^{4-\frac{d}{2}} R_1^{\frac{d}{2}})\Big)\exp\big((1+\sqrt{T}+T)C_1(R_2)\big),\\
\leq &\, C_{p,1}  (\eta_0^2 +C_f T (\eta^4+ \eta^{4-\frac{d}{2}} R_1^{\frac{d}{2}}))\exp\big((1+\sqrt{T}+T)C_1(R_2)\big),
\end{aligned}
\end{equation}   
which implies 
\begin{equation} \label{p_eta}
\Vert p \Vert_{\mathcal{X}_p^1} \leq  \eta
\end{equation}
as long as $0<\eta_0\ll \eta$ is chosen small enough and $\eta$ is reduced as needed. Consequently, we can ascertain that \eqref{p_theta_R} and  \eqref{p_eta} hold by suitably selecting $\eta_0,\eta, R_1, R_2>0$. This thus guarantees that the ball $B$ is invariant with respect to the mapping $\mathcal{T}$. This concludes the proof of Lemma \ref{Lemma_Self_Mapping}. 
\end{proof}
 
\textbf{Step 3. Showing the contraction Property.}
Having established  that  $\mathcal{T}$ is a self-mapping, we now proceed to  verify  the remaining hypotheses necessary  for  applying   the Banach contraction mapping theorem.  
More precisely, we need to ensure that the  mapping $\mathcal{T}$ is a contraction in a suitably chosen topology. Further, for this theorem to apply, the subset $B$ has to be closed in the aforementioned topology. Inspired by \cite{Kaltenbacher_Nikolic_2022}, we consider the metric space  
\begin{equation}
\textbf{Y}:=\big\{(p, \Theta): p \in \mathcal{X}_p^1, \, \Theta  \in L^\infty(0, T; H^2(\Om) \cap H^1_0(\Om)),\, \Theta_t \in L^\infty(0, T; H^1_0(\Om)) \big\}
\end{equation}
equipped with the metric induced by the norm
\begin{equation} \label{norm_Y}
\| (p, \Theta) \|_{\textbf{Y}}= \Big( \|p \|_{L^\infty H^1}^2+\|p_t \|_{L^\infty H^1}^2+ \|p_{tt} \|_{L^\infty L^2}^2+ \| \Theta \|_{L^\infty H^2}^2+\| \Theta_t \|_{L^\infty H^1}^2 \Big)^{\frac{1}{2}}. 
\end{equation} 
Clearly $\textbf{Y}$ is complete and $B \subset \textbf{Y}$. Moreover, $B$ is complete with respect to the norm \eqref{norm_Y}. To prove this claim, consider a sequence $\big((p_n, \Theta_n)\big)_{n \in \mathbb{N}} \subset B$ that is a Cauchy sequence in $\textbf{Y}$. The completeness of the latter implies that this sequence converges in $\textbf{Y}$ to some limit $(p, \Theta) \in \textbf{Y}$. Further, taking into account the boundedness of elements of the ball $B$, there exists a subsequence of $\big((p_n, \Theta_n)\big)_{n \in \mathbb{N}}$ which weakly-$\ast$ converges to $(\tilde{p}, \tilde{\Theta}) \in B$. From the uniqueness of the limit, we obtain $(p, \Theta)= (\tilde{p}, \tilde{\Theta}) \in B$.

In the following, we will prove that the mapping $\mathcal{T}: B \rightarrow B$ is strictly contractive with respect to the metric on $\textbf{Y}$.
\begin{lemma} \label{Lemma_Contraction}
Given $\tau>0$. There exist constants $C_1, C_4, C_6>0$ that do not depend on $T$, such that if the final time $T$ is short enough so as to satisfy 
\begin{equation}\label{}
\begin{aligned}
 T R_1^{\frac{d}{2}} &\, \Big(C_{p, 1}((R_1^{2-\frac{d}{2}}+\eta^{2-\frac{d}{2}}+\eta^{4-\frac{d}{2}}) C_4(R_2) \exp\big((1+\sqrt{T}+T) C_1(R_2)\big)\\
&+ C_{\Theta, 1}(\eta^{2-\frac{d}{2}}+\eta^{4-\frac{d}{2}}) C_6(R_1, R_2) \Big) <1
\end{aligned}
\end{equation}
then the mapping $\mathcal{T}$ is strictly contractive on $B$ with respect to the $\textbf{Y}$ topology. The constants $C_{p, 1}, C_{\Theta, 1}$ are as in estimates \eqref{Main_First_Order_Est} and  \eqref{estimate_theta_H2_1}.
\end{lemma}
\begin{proof}
Let $(p^\ast_1, \Theta^\ast_1), (p^\ast_2, \Theta^\ast_2) \in B$ and let $(p_1, \Theta_1), (p_2, \Theta_2)$ be the corresponding solutions to \eqref{fixed_point_sys}, meaning that 
\begin{equation}   
\mathcal{T}(p^\ast_1, \Theta^\ast_1)=(p_1, \Theta_1)\quad \text{and}\quad  \mathcal{T}(p^\ast_2, \Theta^\ast_2)=(p_2, \Theta_2).
\end{equation}
We can readily check that
\begin{equation}
\begin{aligned}
&\hat{p}=p_1-p_2, \qquad \hat{\Theta}=\Theta_1-\Theta_2,\\
&\hat{p}^\ast=p_1^\ast-p_2^\ast, \qquad \hat{\Theta}^\ast=\Theta_1^\ast-\Theta_2^\ast\\
\end{aligned}
\end{equation}
solve the following system
\begin{equation} \label{contraction_sys}
\left\{ 
\begin{aligned}
&\tau \hat{p}_{ttt}+\hat{p}_{tt}-h (\Theta_1^\ast)\Delta \hat{p} - \zeta(\Theta_1^\ast) \Delta \hat{p}_t = f_1, & \qquad &\\
& m \hat{\Theta}_t -\kappaa \Delta \hat{\Theta} + \ell \hat{\Theta} = g_1,
\end{aligned}
\text{in}\quad  \Omega \times (0,T),
\right.
\end{equation}
with homogeneous initial and boundary conditions
\begin{equation}
\begin{aligned}
&\hat{p}(x,0)=\hat{p}_t(x,0)=\hat{\Theta}(x,0)=0,  &\qquad & \text{in} \quad \ \Omega, \\
&\hat{p}=\hat{\Theta}=0, &\qquad &\text{on}\quad   \partial\Omega \times (0,T).
\end{aligned}
\end{equation}
The source terms $f_1$ and $g_1$  are given, respectively  as  
\begin{equation}\label{}
\begin{aligned}
f_1=&\,\big(h (\Theta_1^\ast)-h (\Theta_2^\ast)\big)\Delta p_2+\big(\zeta (\Theta_1^\ast)-\zeta (\Theta_2^\ast) \big) \Delta p_{2t} \\
&+2k (\Theta_1^\ast)\big((p^\ast_{1t})^2+p_1^\ast p_{1tt}^\ast \big)-2k (\Theta_2^\ast)\big((p^\ast_{2t})^2+p_2^\ast p_{2tt}^\ast \big)\\
=& \, (h (\Theta_1^\ast)-h (\Theta_2^\ast))\Delta p_2+\big(\zeta (\Theta_1^\ast)-\zeta (\Theta_2^\ast) \big) \Delta p_{2t} \\
&+2(k (\Theta_1^\ast)-k (\Theta_2^\ast))((p^\ast_{1t})^2+p_1^\ast p_{1tt}^\ast)+2k (\Theta_2^\ast)\big((p^\ast_{1t})^2-(p^\ast_{2t})^2+p_1^\ast p_{1tt}^\ast-p_2^\ast p_{2tt}^\ast\big)\\
:=&\,f_{11}+f_{12}+f_{13},
\end{aligned}
\end{equation}
with 
\begin{equation}
\left\{
\begin{aligned}
f_{11}=&\,(h (\Theta_1^\ast)-h (\Theta_2^\ast))\Delta p_2+\big(\zeta (\Theta_1^\ast)-\zeta (\Theta_2^\ast) \big) \Delta p_{2t},\\
f_{12}=&\, 2(k (\Theta_1^\ast)-k (\Theta_2^\ast))((p^\ast_{1t})^2+p_1^\ast p_{1tt}^\ast),\\
f_{13}=&\,2k (\Theta_2^\ast)\big((p^\ast_{1t})^2-(p^\ast_{2t})^2+p_1^\ast p_{1tt}^\ast-p_2^\ast p_{2tt}^\ast\big),
\end{aligned}
\right.
\end{equation}
and 
\begin{equation}
g_1=\Phi(\Theta_{1t}^\ast, p_{1t}^\ast)-\Phi(\Theta_{2t}^\ast, p_{2t}^\ast). 
\end{equation}
Notice that the presence of the higher-order derivatives in the component $f_{11}$ prevents from getting the contractivity of the mapping $\mathcal{T}$ with respect to the metric on $\textbf{X}$. Precisely, $(p^\ast_i, \Theta^\ast_i), (p_i, \Theta_i), i=1, 2$ belonging to $B$ does not give $f_1 \in H^1(0, T; L^2(\Om))$. That is why, in the following, we will show instead that $\mathcal{T}$ is a contraction with respect to the weaker topology on $\textbf{Y}$.

As a first step, we work out appropriate estimates for the source terms $f_1$ and $g_1$. We begin with the three components of $f_1$. We have 
\begin{equation} \label{f_11_estimate}
\begin{aligned}
 \| f_{11} \|_{L^2 L^2}=&\,\| (h (\Theta_1^\ast)-h (\Theta_2^\ast))\Delta p_2+\big(\zeta (\Theta_1^\ast)-\zeta (\Theta_2^\ast) \big) \Delta p_{2t} \|_{L^2 L^2}\\
\leq &\,  \|h (\Theta_1^\ast)-h (\Theta_2^\ast)\|_{L^\infty L^\infty} \|  \Delta p_2 \|_{L^2 L^2}\\
&+ \| \zeta (\Theta_1^\ast)-\zeta (\Theta_2^\ast) \|_{L^\infty L^\infty} \|  \Delta p_{2t} \|_{L^2 L^2}.
\end{aligned}
\end{equation}
We obtain an upper bound of the right-hand side by appealing to the properties of the functions $h$ and $\zeta$. Observe that
\begin{equation} \label{h_difference}
h (\Theta_1^\ast)-h (\Theta_2^\ast)=(\Theta_1^\ast-\Theta_2^\ast) \int_0^1 h '(\Theta_2^\ast+\sigma(\Theta_1^\ast-\Theta_2^\ast) ) \textup{d} \sigma.
\end{equation}
Taking into account \eqref{h'_assump} and the embedding $H^2(\Om) \hookrightarrow L^\infty(\Om)$, we find (see \cite{Nikolic_2022} for details)
\begin{equation}\label{h}
\begin{aligned}
\Vert h (\Theta_1^\ast)-h (\Theta_2^\ast) \Vert_{L^\infty L^\infty}&=\Big\Vert (\Theta_1^\ast-\Theta_2^\ast) \int_0^1 h '(\Theta_2^\ast+\sigma(\Theta_1^\ast-\Theta_2^\ast) ) \textup{d} \sigma \Big\Vert_{L^\infty L^\infty}\\
& \lesssim \Vert \Theta_1^\ast-\Theta_2^\ast \Vert_{L^\infty L^\infty} \Big(1+\Vert \Theta_2^\ast+\sigma(\Theta_1^\ast-\Theta_2^\ast) \Vert_{L^\infty L^\infty}^{1+\gamma_1} \Big)\\
& \lesssim \Vert (\hat{p}^\ast, \hat{\Theta}^\ast) \Vert_{\mathbf{Y}} \Big( 1+ \Vert \Theta_1^\ast \Vert_{L^\infty L^\infty}^{1+\gamma_1}+\Vert \Theta_2^\ast \Vert_{L^\infty L^\infty}^{1+\gamma_1}\Big).
\end{aligned}. 
\end{equation}
Similarly, we can show that
\begin{equation}\label{zeta_estimate}
\begin{aligned}
\Vert \zeta (\Theta_1^\ast)-\zeta(\Theta_2^\ast) \Vert_{L^\infty L^\infty} & \lesssim \Vert (\hat{p}^\ast, \hat{\Theta}^\ast) \Vert_{\mathbf{Y}} \Big( 1+ \Vert \Theta_1^\ast \Vert_{L^\infty L^\infty}^{1+\gamma_3}+\Vert \Theta_2^\ast \Vert_{L^\infty L^\infty}^{1+\gamma_3}\Big).
\end{aligned}
\end{equation}
Combining \eqref{h_difference}, \eqref{zeta_estimate} with \eqref{f_11_estimate} yields
\begin{equation} 
\begin{aligned}
 \| f_{11} \|_{L^2 L^2}
\lesssim &\,   \sqrt{T}\|  \Delta p_2 \|_{L^\infty L^2} \Big( 1+ \Vert \Theta_1^\ast \Vert_{L^\infty L^\infty}^{1+\gamma_1}+\Vert \Theta_2^\ast \Vert_{L^\infty L^\infty}^{1+\gamma_1}\Big) \Vert (\hat{p}^\ast, \hat{\Theta}^\ast) \Vert_{\mathbf{Y}} \\
&+  \sqrt{T} \|  \Delta p_{2t} \|_{L^\infty L^2} \Big( 1+ \Vert \Theta_1^\ast \Vert_{L^\infty L^\infty}^{1+\gamma_3}+\Vert \Theta_2^\ast \Vert_{L^\infty L^\infty}^{1+\gamma_3}\Big)  \Vert (\hat{p}^\ast, \hat{\Theta}^\ast) \Vert_{\mathbf{Y}}.
\end{aligned}
\end{equation}
Then recalling the fact that $\mathcal{T}$ maps $B$ into itself, we have
\begin{equation} \label{f_11_estimate_1}
\begin{aligned}
 \| f_{11} \|_{L^2 L^2}
\lesssim &\,   \sqrt{T} R_1 \big( 1+ R_2^{1+\gamma_1}+ R_2^{1+\gamma_3} \big) \Vert (\hat{p}^\ast, \hat{\Theta}^\ast) \Vert_{\mathbf{Y}}.
\end{aligned}
\end{equation}
The same ideas in \eqref{h_difference} and  \eqref{h} can be used again to give an analogous estimate for $k$; namely
\begin{equation}\label{k}
\begin{aligned}
\Vert k (\Theta_1^\ast)-k (\Theta_2^\ast) \Vert_{L^\infty L^\infty} & \lesssim \Vert (\hat{p}^\ast, \hat{\Theta}^\ast) \Vert_{\mathbf{Y}} \Big( 1+ \Vert \Theta_1^\ast \Vert_{L^\infty L^\infty}^{1+\gamma_2}+\Vert \Theta_2^\ast \Vert_{L^\infty L^\infty}^{1+\gamma_2}\Big).
\end{aligned}
\end{equation}
Hence, owing to the fact that $(p^\ast_i , \Theta^\ast_i) \in B, i=1, 2$, we have 
\begin{equation} \label{f_12}
\begin{aligned}
\| f_{12} \|_{L^2 L^2} = &\, \| 2(k (\Theta_1^\ast)-k (\Theta_2^\ast))((p^\ast_{1t})^2+p_1^\ast p_{1tt}^\ast) \|_{L^2 L^2}\\
\leq &\, 2 \| k (\Theta_1^\ast)-k (\Theta_2^\ast) \|_{L^\infty L^\infty} \big(\| p^\ast_{1t} \|_{L^\infty L^\infty} \| p^\ast_{1t} \|_{L^2 L^2} + \| p^\ast_{1} \|_{L^\infty L^\infty} \| p^\ast_{1tt} \|_{L^2 L^2} \big)\\
\lesssim &\, \sqrt{T} \big(\| p^\ast_{1t} \|_{L^\infty H^2}^{\frac{d}{4}} \| p^\ast_{1t} \|_{L^\infty L^2}^{1-\frac{d}{4}} \| p^\ast_{1t} \|_{L^\infty L^2}+\| p^\ast_{1} \|_{L^\infty H^2}^{\frac{d}{4}} \| p^\ast_{1} \|_{L^\infty L^2}^{1-\frac{d}{4}} \| p^\ast_{1tt} \|_{L^\infty L^2} \big)\\
& \times \Big( 1+ \Vert \Theta_1^\ast \Vert_{\mathcal{X}_\Theta^2}^{1+\gamma_2}+\Vert \Theta_2^\ast \Vert_{\mathcal{X}_\Theta^2}^{1+\gamma_2}\Big)\Vert (\hat{p}^\ast, \hat{\Theta}^\ast) \Vert_{\mathbf{Y}}\\
\lesssim &\, \sqrt{T} \| p^\ast_{1} \|_{\mathcal{X}_p^1}^{2-\frac{d}{4}} \| p^\ast_{1} \|_{\mathcal{X}_p^2}^{\frac{d}{4}}  \Big( 1+ \Vert \Theta_1^\ast \Vert_{\mathcal{X}_\Theta^2}^{1+\gamma_2}+\Vert \Theta_2^\ast \Vert_{\mathcal{X}_\Theta^2}^{1+\gamma_2}\Big)\Vert (\hat{p}^\ast, \hat{\Theta}^\ast) \Vert_{\mathbf{Y}}\\
\lesssim &\, \sqrt{T} \eta ^{2-\frac{d}{4}} R_1^{\frac{d}{4}}   (1+R_2^{1+\gamma_2})\Vert (\hat{p}^\ast, \hat{\Theta}^\ast) \Vert_{\mathbf{Y}}.
\end{aligned}
\end{equation}
Note that we reached the above estimate thanks to the interpolation inequality \eqref{Agmon_Inequality} and we have $T$ as  a prefactor from using the embedding $L^\infty(0, T) \hookrightarrow L^2(0, T)$.\\
The last component $f_{13}$ can be written as follows
\begin{equation}
\begin{aligned}
f_{13}=&\, 2k (\Theta_2^\ast)\big((p^\ast_{1t})^2-(p^\ast_{2t})^2+p_1^\ast p_{1tt}^\ast-p_2^\ast p_{2tt}^\ast\big)\\
=&\, 2k (\Theta_2^\ast)\big((p^\ast_{1t}-p^\ast_{2t})(p^\ast_{1t}+p^\ast_{2t})+(p_1^\ast-p_2^\ast) p_{1tt}^\ast+p_2^\ast(p_{1tt}^\ast- p_{2tt}^\ast)\big).
\end{aligned}
\end{equation}
On account of \eqref{k_1} and the embeddings $H^1(\Om) \hookrightarrow L^4(\Om)$, $L^\infty(0, T) \hookrightarrow L^2(0, T)$, we find
\begin{equation}
\begin{aligned}
\| f_{13} \|_{L^2 L^2} \leq &\, 2 \|k (\Theta_2^\ast)\|_{L^\infty L^\infty} \big( \|p^\ast_{1t}-p^\ast_{2t}\|_{L^\infty L^2} \|p^\ast_{1t}+p^\ast_{2t}\|_{L^2 L^\infty}\\
&+\|p_1^\ast-p_2^\ast\|_{L^\infty L^4} \| p_{1tt}^\ast\|_{L^2 L^4}+\| p_2^\ast \|_{L^\infty L^\infty} \| p_{1tt}^\ast- p_{2tt}^\ast\|_{L^2 L^2} \big)\\
\lesssim &\, k_1 \sqrt{T} \big( (\|p^\ast_{1t}\|_{L^\infty L^\infty}+ \|p^\ast_{2t}\|_{L^\infty L^\infty}) \|p^\ast_{1t}-p^\ast_{2t}\|_{L^\infty L^2}\\
&+  \| p_{1tt}^\ast\|_{L^\infty L^4} \|p_1^\ast-p_2^\ast\|_{L^\infty H^1}+ \| p_2^\ast \|_{L^\infty L^\infty} \| p_{1tt}^\ast- p_{2tt}^\ast\|_{L^\infty L^2} \big). 
%\leq &\, C_T k_1 (\| p^\ast_1 \|_{X_p}+\| p^\ast_2 \|_{X_p})\|p_1^\ast-p_2^\ast\|_{\mathfrak{X}_1}.
\end{aligned}
\end{equation}
Further, using Agmon's inequality \eqref{Agmon_Inequality} and Ladyzhenskaya's one \eqref{Lady_Inequality}, the above estimate becomes
\begin{equation}
\begin{aligned}
\| f_{13} \|_{L^2 L^2}  
\lesssim &\,  \sqrt{T} \left( \Big(\sum_{i=1}^2\|p^\ast_{it}\|_{L^\infty L^2}^{1-\frac{d}{4}} \|p^\ast_{it}\|_{L^\infty H^2}^{\frac{d}{4}}\Big) \|p^\ast_{1t}-p^\ast_{2t}\|_{L^\infty H^1}\right.\\
&\left.+  \| p_{1tt}^\ast\|_{L^\infty L^2}^{1-\frac{d}{4}} \| p_{1tt}^\ast\|_{L^\infty H^1}^{\frac{d}{4}}\|p_1^\ast-p_2^\ast\|_{L^\infty H^1}\right.\\
&\left.+ \| p_2^\ast \|_{L^\infty L^2}^{1-\frac{d}{4}} \| p_2^\ast \|_{L^\infty H^2}^{\frac{d}{4}} \| p_{1tt}^\ast- p_{2tt}^\ast\|_{L^\infty L^2} \right)\\
\lesssim &\, \sqrt{T} \Big(\sum_{i=1}^2\| p^\ast_{i} \|_{\mathcal{X}_p^1}^{1-\frac{d}{4}} \| p^\ast_{i} \|_{\mathcal{X}_p^2}^{\frac{d}{4}}  \Big) \| \hat{p}^\ast \|_{\mathcal{X}_p^1}.
\end{aligned}
\end{equation}
Therefore,
\begin{equation} \label{f_13}
\| f_{13} \|_{L^2 L^2} \lesssim \sqrt{T}  \eta^{1-\frac{d}{4}} R_1^{\frac{d}{4}} \Vert (\hat{p}^\ast, \hat{\Theta}^\ast) \Vert_{\mathbf{Y}}.
\end{equation}
Consequently,  from the estimates \eqref{f_11_estimate_1}, \eqref{f_12} and \eqref{f_13}, we conclude
\begin{equation} \label{f_1_estimate}
\| f_{1} \|_{L^2 L^2}^2 \lesssim T R_1^{\frac{d}{2}}(R_1^{2-\frac{d}{2}}+\eta^{2-\frac{d}{2}} + \eta^{4-\frac{d}{2}}) C_4(R_2) \Vert (\hat{p}^\ast, \hat{\Theta}^\ast) \Vert_{\mathbf{Y}}^2.
\end{equation}
Next, we estimate the forcing term in the Pennes equation in \eqref{contraction_sys}. For the sake of clarity, we rewrite the formula for $\Phi$ where $\phi$ is given in \eqref{funct_k}
$$\Phi(\Theta^\ast, p^\ast_t)=\phi(\Theta^\ast) (p_t^\ast)^2.$$
The function $g_1$ can be recast as 
\begin{equation}
\begin{aligned}
g_1= &\, (\phi(\Theta^\ast_1)-\phi(\Theta^\ast_2)) (p_{1t}^\ast)^2+ \phi(\Theta^\ast_2)((p_{1t}^\ast)^2-(p_{2t}^\ast)^2)\\
= &\, (\phi(\Theta^\ast_1)-\phi(\Theta^\ast_2)) (p_{1t}^\ast)^2+ \phi(\Theta^\ast_2)(p_{1t}^\ast-p_{2t}^\ast)(p_{1t}^\ast+p_{2t}^\ast).
\end{aligned}
\end{equation}
We can use the same trick in \eqref{h_difference} to obtain a similar estimate to \eqref{h} for the difference $\phi(\Theta^\ast_1)-\phi(\Theta^\ast_2)$. Indeed, thanks to \eqref{phi_inequalities}, we find the following bound
\begin{equation} \label{phi_difference}
\begin{aligned}
\| \phi(\Theta^\ast_1)-\phi(\Theta^\ast_2) \|_{L^\infty L^\infty} \lesssim \Vert (\hat{p}^\ast, \hat{\Theta}^\ast) \Vert_{\mathbf{Y}} \Big( 1+ \Vert \Theta_1^\ast \Vert_{L^\infty L^\infty}^{1+\gamma_4}+\Vert \Theta_2^\ast \Vert_{L^\infty L^\infty}^{1+\gamma_4}\Big),
\end{aligned}
\end{equation}
which along with \eqref{phi_1} the embedding $L^\infty(0, T) \hookrightarrow L^2(0, T)$ will allow us to get 
\begin{equation}
\begin{aligned}
\Vert g_1 \Vert_{L^2 L^2}=&\,\| (\phi(\Theta^\ast_1)-\phi(\Theta^\ast_2)) (p_{1t}^\ast)^2+ \phi(\Theta^\ast_2)(p_{1t}^\ast-p_{2t}^\ast)(p_{1t}^\ast+p_{2t}^\ast) \|_{L^2 L^2}\\
\leq  &\, \| \phi(\Theta^\ast_1)-\phi(\Theta^\ast_2) \|_{L^\infty L^\infty} \| p_{1t}^\ast \|_{L^\infty L^\infty} \| p_{1t}^\ast\|_{L^2 L^2} \\
&\,+ \|\phi(\Theta^\ast_2) \|_{L^\infty L^\infty} \|p_{1t}^\ast-p_{2t}^\ast \|_{L^2 L^2} \|p_{1t}^\ast+p_{2t}^\ast \|_{L^\infty L^\infty}\\
\lesssim &\, \sqrt{T} \| p_{1t}^\ast \|_{L^\infty L^\infty} \| p_{1t}^\ast\|_{L^\infty L^2} \Big( 1+ \Vert \Theta_1^\ast \Vert_{L^\infty L^\infty}^{1+\gamma_4}+\Vert \Theta_2^\ast \Vert_{L^\infty L^\infty}^{1+\gamma_4}\Big) \Vert (\hat{p}^\ast, \hat{\Theta}^\ast) \Vert_{\mathbf{Y}}\\
&\,+ \sqrt{T} \phi_1 ( \|p_{1t}^\ast \|_{L^\infty L^\infty}+ \|p_{2t}^\ast \|_{L^\infty L^\infty}) \|p_{1t}^\ast-p_{2t}^\ast \|_{L^\infty H^1}.
\end{aligned}
\end{equation}
Applying inequality \eqref{Agmon_Inequality} yields
\begin{equation}\label{estimate_g_1}
\begin{aligned}
&\Vert g_1 \Vert_{L^2 L^2}\\ 
\lesssim &\, \sqrt{T} \| p_{1t}^\ast \|_{L^\infty L^2}^{1-\frac{d}{4}} \| p_{1t}^\ast \|_{L^\infty H^2}^{\frac{d}{4}} \| p_{1t}^\ast\|_{L^\infty L^2} \Big( 1+ \Vert \Theta_1^\ast \Vert_{L^\infty L^\infty}^{1+\gamma_4}+\Vert \Theta_2^\ast \Vert_{L^\infty L^\infty}^{1+\gamma_4}\Big) \Vert (\hat{p}^\ast, \hat{\Theta}^\ast) \Vert_{\mathbf{Y}}\\
&\,+ \sqrt{T} \phi_1 \big(\|p_{1t}^\ast \|_{L^\infty L^2}^{1-\frac{d}{4}}\|p_{1t}^\ast \|_{L^\infty H^2}^{\frac{d}{4}}+\|p_{2t}^\ast \|_{L^\infty L^2}^{1-\frac{d}{4}} \| p_{2t}^\ast \|_{L^\infty H^2}^{\frac{d}{4}} \big) \|\hat{p}^\ast \|_{\mathcal{X}_p^1}\\  
\lesssim &\, \sqrt{T} \eta^{1-\frac{d}{4}} R_1^{\frac{d}{4}}(1+R_1+R_1R_2^{1+\gamma_4})\Vert (\hat{p}^\ast, \hat{\Theta}^\ast) \Vert_{\mathbf{Y}}.
\end{aligned}
\end{equation}
Likewise, we have
\begin{equation}
\begin{aligned}
g_{1t}= &\, \phi'(\Theta^\ast_1) \Theta^\ast_{1t}(p_{1t}^\ast)^2-\phi'(\Theta^\ast_2) \Theta^\ast_{2t}(p_{2t}^\ast)^2+2\phi(\Theta^\ast_1) p_{1t}^\ast p_{1tt}^\ast-2\phi(\Theta^\ast_2) p_{2t}^\ast p_{2tt}^\ast\\
= &\, 
\begin{multlined}[t]
\phi'(\Theta^\ast_1) \Theta^\ast_{1t}(p_{1t}^\ast-p_{2t}^\ast)(p_{1t}^\ast+p_{2t}^\ast)+\phi'(\Theta^\ast_1)(p_{2t}^\ast)^2( \Theta^\ast_{1t}-\Theta^\ast_{2t})\\
+(\phi'(\Theta^\ast_1)-\phi'(\Theta^\ast_2))\Theta^\ast_{2t}(p_{2t}^\ast)^2 +2 \phi(\Theta^\ast_1)p_{1t}^\ast(p_{1tt}^\ast-p_{2tt}^\ast)\\
+ 2 \phi(\Theta^\ast_1)(p_{1t}^\ast-p_{2t}^\ast)p_{2tt}^\ast+ 2 (\phi(\Theta^\ast_1)-\phi(\Theta^\ast_2)) p_{2t}^\ast p_{2tt}^\ast\\
\vspace{-0.4cm}
\end{multlined}\\
:=&\,g_{1t}^{(1)}+g_{1t}^{(2)}+g_{1t}^{(3)}
\end{aligned}
\end{equation}
with
\begin{equation}
\left\{\hspace{0.1cm}
\begin{aligned}
g_{1t}^{(1)}=&\,\phi'(\Theta^\ast_1) \Theta^\ast_{1t}(p_{1t}^\ast-p_{2t}^\ast)(p_{1t}^\ast+p_{2t}^\ast)+\phi'(\Theta^\ast_1)(p_{2t}^\ast)^2( \Theta^\ast_{1t}-\Theta^\ast_{2t});\\
g_{1t}^{(2)}=&\, 2 \phi(\Theta^\ast_1)p_{1t}^\ast(p_{1tt}^\ast-p_{2tt}^\ast)+2 \phi(\Theta^\ast_1)(p_{1t}^\ast-p_{2t}^\ast)p_{2tt}^\ast;\\
g_{1t}^{(3)}=&\,(\phi'(\Theta^\ast_1)-\phi'(\Theta^\ast_2))\Theta^\ast_{2t}(p_{2t}^\ast)^2+2 (\phi(\Theta^\ast_1)-\phi(\Theta^\ast_2)) p_{2t}^\ast p_{2tt}^\ast.
\end{aligned}
\right. 
\end{equation}
We begin by estimating the first component $g_{1t}^{(1)}$. By using H\"{o}lder's inequality together with \eqref{phi_inequalities} and the embeddings $H^1(\Om) \hookrightarrow L^4(\Om)$, $L^\infty(0, T) \hookrightarrow L^2(0, T)$, it follows
\begin{equation}
\begin{aligned}
\Vert g_{1t}^{(1)} \Vert_{L^2 L^2} \leq &\, \| \phi'(\Theta^\ast_1) \|_{L^\infty L^\infty} \|\Theta^\ast_{1t} \|_{L^\infty L^4} \|p_{1t}^\ast-p_{2t}^\ast\|_{L^2 L^4} \|p_{1t}^\ast+p_{2t}^\ast\|_{L^\infty L^\infty}\\
& +\|\phi'(\Theta^\ast_1)\| _{L^\infty L^\infty} \|p_{2t}^\ast \|_{L^\infty L^\infty} \|p_{2t}^\ast \|_{L^2 L^4} \| \Theta^\ast_{1t}-\Theta^\ast_{2t} \|_{L^\infty L^4}\\
\lesssim &\, \sqrt{T}(1+ \Vert \Theta_1^\ast \Vert_{L^\infty L^\infty}^{1+\gamma_4}) \|\Theta^\ast_{1t} \|_{L^\infty H^1} (\|p_{1t}^\ast\|_{L^\infty L^\infty}+ \|p_{2t}^\ast\|_{L^\infty L^\infty}) \|p_{1t}^\ast-p_{2t}^\ast\|_{L^\infty H^1}\\
&+\sqrt{T}(1+ \Vert \Theta_1^\ast \Vert_{L^\infty L^\infty}^{1+\gamma_4}) \|p_{2t}^\ast \|_{L^\infty L^\infty} \|p_{2t}^\ast \|_{L^\infty H^1} \| \Theta^\ast_{1t}-\Theta^\ast_{2t} \|_{L^\infty H^1}.
\end{aligned}
\end{equation}
We apply \eqref{Agmon_Inequality} to further estimate the right-hand side of the inequality above. Thus, we have 
 \begin{equation} \label{g_11}
\begin{aligned}
\Vert g_{1t}^{(1)} \Vert_{L^2 L^2} 
\lesssim &\, \sqrt{T}(1+ \Vert \Theta_1^\ast \Vert_{L^\infty L^\infty}^{1+\gamma_4}) \|\Theta^\ast_{1t} \|_{L^\infty H^1} ( \sum_{i=1}^2\|p_{it}^\ast\|_{L^\infty L^2}^{1-\frac{d}{4}} \|p_{it}^\ast\|_{L^\infty H^2}^{\frac{d}{4}}) \|\hat{p}^\ast\|_{\mathcal{X} _p^1}\\
&+\sqrt{T} \|p_{2t}^\ast\|_{L^\infty L^2}^{1-\frac{d}{4}} \|p_{2t}^\ast\|_{L^\infty H^2}^{\frac{d}{4}} \|p_{2t}^\ast \|_{L^\infty H^1} (1+ \Vert \Theta_1^\ast \Vert_{L^\infty L^\infty}^{1+\gamma_4}) \| \hat{\Theta}^\ast_{t} \|_{L^\infty H^1}\\
\lesssim &\, \sqrt{T} \eta^{1-\frac{d}{4}} R_1^{\frac{d}{4}}(1+R_1^{1+\gamma_4})(R_1+R_2)\Vert (\hat{p}^\ast, \hat{\Theta}^\ast) \Vert_{\mathbf{Y}}.  
\end{aligned}
\end{equation}
The last bound above follows using the fact that $(p_{i}^\ast, \Theta_{i}^\ast) \in B,\, i=1, 2$.\\
Similarly, we obtain an estimate for $g_{1t}^{(2)}$. Thanks to \eqref{phi_1} and the same embeddings, we have
\begin{equation}
\begin{aligned}
\Vert g_{1t}^{(2)} \Vert_{L^2 L^2} \leq &\, 2 \| \phi(\Theta^\ast_1) \|_{L^\infty L^\infty} \|p_{1t}^\ast \|_{L^\infty L^\infty} \|p_{1tt}^\ast-p_{2tt}^\ast \|_{L^2 L^2}\\
&+ 2 \|\phi(\Theta^\ast_1) \|_{L^\infty L^\infty} \|p_{1t}^\ast-p_{2t}^\ast \|_{L^\infty L^4} \|p_{2tt}^\ast \|_{L^2 L^4}\\
\lesssim &\, \sqrt{T} \phi_1 \|p_{1t}^\ast \|_{L^\infty L^\infty} \|p_{1tt}^\ast-p_{2tt}^\ast \|_{L^\infty L^2}\\
&+  \sqrt{T} \phi_1   \|p_{2tt}^\ast \|_{L^\infty L^4} \|p_{1t}^\ast-p_{2t}^\ast \|_{L^\infty H^1}. \\
\end{aligned}
\end{equation} 
Again, taking advantage of \eqref{Agmon_Inequality} and \eqref{Lady_Inequality},  gives
 \begin{equation} \label{g_12}
\begin{aligned}
\Vert g_{1t}^{(2)} \Vert_{L^2 L^2} \lesssim &\, \sqrt{T} \Big(\|p_{1t}^\ast\|_{L^\infty L^2}^{1-\frac{d}{4}} \|p_{1t}^\ast\|_{L^\infty H^2}^{\frac{d}{4}}+ \|p_{2tt}^\ast\|_{L^\infty L^2}^{1-\frac{d}{4}} \|p_{2tt}^\ast\|_{L^\infty H^1}^{\frac{d}{4}} \Big) \|\hat{p}^\ast\|_{\mathcal{X} _p^1}\\
\lesssim &\, \sqrt{T} \eta^{1-\frac{d}{4}} R_1^{\frac{d}{4}} \Vert (\hat{p}^\ast, \hat{\Theta}^\ast) \Vert_{\mathbf{Y}}.
\end{aligned}
\end{equation}
The one remaining contribution $g_{1t}^{(3)}$ can be treated in a similar way. First, by relying on the same tools  as above, we have 
\begin{equation} \label{g1_component_3}
\begin{aligned}
\Vert g_{1t}^{(3)} \Vert_{L^2 L^2}\leq &\, 
\|\phi'(\Theta^\ast_1)-\phi'(\Theta^\ast_2) \|_{L^\infty L^\infty}  \| \Theta^\ast_{2t} \|_{L^\infty L^4} \|p_{2t}^\ast \|_{L^\infty L^\infty} \|p_{2t}^\ast \|_{L^2 L^4}\\
& + 2 \|\phi(\Theta^\ast_1)-\phi(\Theta^\ast_2) \|_{L^\infty L^\infty} \| p_{2t}^\ast \|_{L^\infty L^\infty} \|p_{2tt}^\ast\|_{L^2 L^2}\\
\lesssim &\, \sqrt{T}\|\phi'(\Theta^\ast_1)-\phi'(\Theta^\ast_2) \|_{L^\infty L^\infty}  \| \Theta^\ast_{2t} \|_{L^\infty H^1} \|p_{2t}^\ast \|_{L^\infty L^\infty} \|p_{2t}^\ast \|_{L^\infty H^1}\\
& + \sqrt{T} \|\phi(\Theta^\ast_1)-\phi(\Theta^\ast_2) \|_{L^\infty L^\infty} \| p_{2t}^\ast \|_{L^\infty L^\infty} \|p_{2tt}^\ast\|_{L^\infty L^2}.
\end{aligned}
\end{equation}
This shows that we need an estimate for the difference $\phi'(\Theta^\ast_1)-\phi'(\Theta^\ast_2)$. To acquire such a bound, we write
\begin{equation} \label{}
\phi' (\Theta_1^\ast)-\phi'(\Theta_2^\ast)=(\Theta_1^\ast-\Theta_2^\ast) \int_0^1 \phi''(\Theta_2^\ast+\sigma(\Theta_1^\ast-\Theta_2^\ast) ) \textup{d} \sigma.
\end{equation}
On account of \eqref{phi_inequalities} and the embedding $H^2(\Om) \hookrightarrow L^\infty(\Om)$, we have 
\begin{equation}\label{phi_derivative}
\begin{aligned}
\Vert \phi' (\Theta_1^\ast)-\phi'(\Theta_2^\ast) \Vert_{L^\infty L^\infty}&=\Big\Vert (\Theta_1^\ast-\Theta_2^\ast) \int_0^1 \phi''(\Theta_2^\ast+\sigma(\Theta_1^\ast-\Theta_2^\ast) ) \textup{d} \sigma \Big\Vert_{L^\infty L^\infty}\\
& \lesssim \Vert \Theta_1^\ast-\Theta_2^\ast \Vert_{L^\infty L^\infty} \Big(1+\Vert \Theta_2^\ast+\sigma(\Theta_1^\ast-\Theta_2^\ast) \Vert_{L^\infty L^\infty}^{\gamma_4} \Big)\\
& \lesssim \Vert (\hat{p}^\ast, \hat{\Theta}^\ast) \Vert_{\mathbf{Y}} \Big( 1+ \Vert \Theta_1^\ast \Vert_{L^\infty L^\infty}^{\gamma_4}+\Vert \Theta_2^\ast \Vert_{L^\infty L^\infty}^{\gamma_4}\Big).
\end{aligned}. 
\end{equation}
By incorporating \eqref{phi_difference} and  \eqref{phi_derivative} into  the estimate \eqref{g1_component_3}, we find
\begin{equation} \label{}
\begin{aligned}
& \Vert g_{1t}^{(3)} \Vert_{L^2 L^2}\\
\lesssim  
& \sqrt{T}\| \Theta^\ast_{2t} \|_{L^\infty H^1} \|p_{2t}^\ast \|_{L^\infty L^\infty} \|p_{2t}^\ast \|_{L^\infty H^1} \Big( 1+ \Vert \Theta_1^\ast \Vert_{L^\infty L^\infty}^{\gamma_4}+\Vert \Theta_2^\ast \Vert_{L^\infty L^\infty}^{\gamma_4}\Big)\Vert (\hat{p}^\ast, \hat{\Theta}^\ast) \Vert_{\mathbf{Y}} \\
&+ \sqrt{T}  \| p_{2t}^\ast \|_{L^\infty L^\infty} \|p_{2tt}^\ast\|_{L^\infty L^2} \Big( 1+ \Vert \Theta_1^\ast \Vert_{L^\infty L^\infty}^{1+\gamma_4}+\Vert \Theta_2^\ast \Vert_{L^\infty L^\infty}^{1+\gamma_4}\Big) \Vert (\hat{p}^\ast, \hat{\Theta}^\ast) \Vert_{\mathbf{Y}}. 
\end{aligned}
\end{equation}
Hence, from \eqref{Agmon_Inequality} and the fact that $(p_{i}^\ast, \Theta_{i}^\ast) \in B, i=1, 2$, we deduce
\begin{equation} \label{g_13}
\begin{aligned}
& \Vert g_{1t}^{(3)} \Vert_{L^2 L^2}\\
\lesssim  
& \sqrt{T}\| \Theta^\ast_{2t} \|_{L^\infty H^1} \|p_{2t}^\ast\|_{L^\infty L^2}^{1-\frac{d}{4}} \|p_{2t}^\ast\|_{L^\infty H^2}^{\frac{d}{4}} \|p_{2t}^\ast \|_{L^\infty H^1} \\
&\,\times\Big( 1+ \Vert \Theta_1^\ast \Vert_{L^\infty L^\infty}^{\gamma_4}+\Vert \Theta_2^\ast \Vert_{L^\infty L^\infty}^{\gamma_4}\Big)\Vert (\hat{p}^\ast, \hat{\Theta}^\ast) \Vert_{\mathbf{Y}} \\
&+ \sqrt{T} \Big( 1+ \Vert \Theta_1^\ast \Vert_{L^\infty L^\infty}^{1+\gamma_4}+\Vert \Theta_2^\ast \Vert_{L^\infty L^\infty}^{1+\gamma_4}\Big) \|p_{2t}^\ast\|_{L^\infty L^2}^{1-\frac{d}{4}} \|p_{2t}^\ast\|_{L^\infty H^2}^{\frac{d}{4}} \|p_{2tt}^\ast\|_{L^\infty L^2} \Vert (\hat{p}^\ast, \hat{\Theta}^\ast) \Vert_{\mathbf{Y}}\\
\lesssim &\, \sqrt{T} \eta^{2-\frac{d}{4}} R_1^{\frac{d}{4}}(1+R_2+R_2^{1+\gamma_4})\Vert (\hat{p}^\ast, \hat{\Theta}^\ast) \Vert_{\mathbf{Y}}.
\end{aligned}
\end{equation}
By putting together the estimates \eqref{g_11}, \eqref{g_12} and \eqref{g_13}, we infer that
\begin{equation} \label{estimate_g_1t}
\begin{aligned}
 \Vert g_{1t} \Vert_{L^2 L^2}^2 \lesssim &\, \Vert g_{1t}^{(1)} \Vert_{L^2 L^2}^2+\Vert g_{1t}^{(2)} \Vert_{L^2 L^2}^2+\Vert g_{1t}^{(3)} \Vert_{L^2 L^2}^2\\
\leq &\, T R_1^{\frac{d}{2}} (\eta^{2-\frac{d}{2}}+\eta^{4-\frac{d}{2}}) C_5(R_1, R_2)\Vert (\hat{p}^\ast, \hat{\Theta}^\ast) \Vert_{\mathbf{Y}}^2.
\end{aligned} 
\end{equation}
We emphasize that the constant $C_5$ depends on $R_1, R_2$ but not on $T$ and $\eta$.

Now, we call on two estimates from Propositions \ref{lemma_lower_norms}, \ref{prop2}. We know from the proof of the self-mapping property that the coefficients $h(\Theta^\ast_1), \zeta(\Theta^\ast_1)$ in the problem \eqref{contraction_sys} fulfill  the assumptions of Proposition \ref{lemma_lower_norms}, and that the function $\Lambda_0$ in \eqref{lambda_0} satisfies 
\begin{equation}
\| \Lambda_0 \|_{L^1(0, T)} \leq (1+\sqrt{T}+T) C_1(R_2).
\end{equation}
Then the estimate \eqref{Main_First_Order_Est} holds for the solution $\hat{p}$, that is
\begin{equation}
\begin{aligned}
\| \hat{p} \|_{\mathcal{X}^1_p}^2 \leq &\, C_{p, 1} \int_0^t \| f_1(s) \|_{L^2}^2 \exp \Big(  \int_s^t \Lambda_0(r) \textup{d} r \Big) \ds\\
\leq &\, C_{p, 1} T R_1^{\frac{d}{2}}(R_1^{2-\frac{d}{2}}+\eta^{2-\frac{d}{2}}+\eta^{4-\frac{d}{2}}) C_4(R_2) \exp\big((1+\sqrt{T}+T) C_1(R_2)\big) \Vert (\hat{p}^\ast, \hat{\Theta}^\ast) \Vert_{\mathbf{Y}}^2.
\end{aligned}
\end{equation}
The last bound above is attained by making use of \eqref{f_1_estimate}.\\
On the other hand, recall that $\Theta^\ast_1(0)=\Theta^\ast_2(0)=\Theta_0$, $p^\ast_{1t}(0)=p^\ast_{2t}(0)=p_1$, which implies that $g_1(x,0)=0$. Thus, the estimate \eqref{estimate_theta_H2_1} in the proof of Proposition \ref{prop2}, along with \eqref{estimate_g_1} and  \eqref{estimate_g_1t} gives
\begin{equation}
\begin{aligned}
\|  \hat{\Theta} \|_{L^\infty H^2}^2+ \|  \hat{\Theta}_t \|_{L^\infty H^1}^2 \leq &\, C_{\Theta, 1} \|g_1 \|_{H^1 L^2}^2\\
\leq &\, C_{\Theta, 1} T R_1^{\frac{d}{2}} (\eta^{2-\frac{d}{2}}+\eta^{4-\frac{d}{2}})C_6(R_1, R_2)\Vert (\hat{p}^\ast, \hat{\Theta}^\ast) \Vert_{\mathbf{Y}}^2, 
\end{aligned}
\end{equation}
where the constant $C_6$ does not depend on the final time $T$.\\
Consequently, we obtain  
\begin{equation}
\begin{aligned}
 \| \mathcal{T}(p^\ast_1, \Theta^\ast_1) &\,-\mathcal{T}(p^\ast_2, \Theta^\ast_2) \|_{\mathbf{Y}}^2\\
=&\, \| (\hat{p}, \hat{\Theta}) \|_{\mathbf{Y}}^2\\
=&\, \| \hat{p} \|_{\mathcal{X}^1_p}^2+\|  \hat{\Theta} \|_{L^\infty H^2}^2+ \|  \hat{\Theta}_t \|_{L^\infty H^1}^2\\
\leq &\, \Big(C_{p, 1} T R_1^{\frac{d}{2}}(R_1^{2-\frac{d}{2}}+\eta^{2-\frac{d}{2}}+\eta^{4-\frac{d}{2}}) C_4(R_2) \exp\big((1+\sqrt{T}+T) C_1(R_2)\big)\\
&+ C_{\Theta, 1} T R_1^{\frac{d}{2}} (\eta^{2-\frac{d}{2}}+\eta^{4-\frac{d}{2}}) C_6(R_1, R_2) \Big) \Vert (\hat{p}^\ast, \hat{\Theta}^\ast) \Vert_{\mathbf{Y}}^2,
\end{aligned}
\end{equation}
which shows that to ensure that the mapping $\mathcal{T}$ is a contraction, the final time $T$ has to be short enough so as 
\begin{equation}\label{Inequalti_smallness_T}
\begin{aligned}
 T R_1^{\frac{d}{2}} &\, \Big(C_{p, 1}((R_1^{2-\frac{d}{2}}+\eta^{2-\frac{d}{2}}+\eta^{4-\frac{d}{2}}) C_4(R_2) \exp\big((1+\sqrt{T}+T) C_1(R_2)\big)\\
&+ C_{\Theta, 1}(\eta^{2-\frac{d}{2}}+\eta^{4-\frac{d}{2}}) C_6(R_1, R_2) \Big) <1
\end{aligned}
\end{equation}
Lastly, by going over the proof, we can see that the prefactor $R_1^{\frac{d}{2}}$ is due to the higher-order norms of the acoustic pressure. Thus, the above inequality could also be true if the final time $T$ is fixed and the smallness is imposed on the size of the higher-order norms of the pressure $p$. From the practical point of view, it would be better to assume smallness on the lower norms (i.e., on $\eta$), yet this seems unattainable in the current setting owing to the presence of the higher-order derivatives in $f_{11}$. This completes the proof of Lemma \ref{Lemma_Contraction}.
\end{proof}
As a consequence of the two Lemmas \ref{Lemma_Self_Mapping}, \ref{Lemma_Contraction}, the mapping $\mathcal{T}$ is invariant with respect to the subset $B$ and is a contraction in the $\textbf{Y}$ topology. In addition, $B$ is closed in $\textbf{Y}$. Thus based on Banach fixed-point theorem, the mapping $\mathcal{T}$ admits a unique fixed-point $(p, \Theta) \in B$ which, by looking at the definition of $\mathcal{T}$, implies that $(p, \Theta)$ is the unique solution to \eqref{Main_system_JMGT} in $B$. Furthermore, the solution depends continuously on the data. To see this, we appeal to the following estimates \eqref{interm_estimate}, \eqref{Regularity_H^2_Heat} written as follows
\begin{equation} \label{est_p}
\begin{aligned}
& \|p(t) \|_{H^2}^2+\|p_t(t) \|_{H^2}^2+ \tau \|p_{tt}(t) \|_{H^1}^2 \\
\lesssim &\, \|p_0 \|_{H^2}^2+\|p_1 \|_{H^2}^2+ \tau \|p_{2} \|_{H^1}^2+ \int_0^t \Lambda(s) \big(\|p(s) \|_{H^2}^2+\|p_t(s) \|_{H^2}^2+ \tau \|p_{tt}(s) \|_{H^1}^2\big) \ds\\
&+\int_0^t (\|f(s) \|^2_{L^2}+\|\nabla f(s) \|^2_{L^2}) \ds ,
\end{aligned}
\end{equation} 
and 
\begin{equation} \label{est_theta}
\begin{aligned}
&  \|\Theta(t) \|_{H^2}^2 +\int_0^t \|\Theta(s) \|_{H^3}^2 \ds+\|\Theta_t(t) \|_{H^1}^2+ \int_0^t \|\Theta_{t}(s) \|_{H^2}^2 \ds \\
\lesssim &\,  \|\Theta_0 \|_{H^3}^2+\|\nabla g(0) \|_{L^2}^2+ \int_0^t (\|\nabla g(s) \|^2_{L^2}+\| g_t(s) \|^2_{L^2}) \ds 
\end{aligned}
\end{equation}
where $0 \leq t \leq T$.\\

Going back to \eqref{g(0)_estimate}, \eqref{g_estimate_0}, \eqref{g_estimate} and noting that now we have $(p^\ast, \Theta^\ast)=(p, \Theta) \in B$, we can see that
\begin{equation} \label{est_theta_1}
\begin{aligned}
&  \|\Theta(t) \|_{H^2}^2 +\int_0^t \|\Theta(s) \|_{H^3}^2 \ds+\|\Theta_t(t) \|_{H^1}^2+ \int_0^t \|\Theta_{t}(s) \|_{H^2}^2 \ds \\
\leq &\, C_\Theta(R_1, R_2) \Big( \|\Theta_0 \|_{H^3}^2+\|p_1 \|_{H^2}^2+ \int_0^t \big(\|p(s) \|_{H^2}^2+\|p_t(s) \|_{H^2}^2+ \tau \|p_{tt}(s) \|_{H^1}^2\big) \ds \Big).
\end{aligned}
\end{equation}
Similarly, by using the estimates \eqref{estimate_f}, \eqref{gradient_f_est_0}, \eqref{gradient_f_est}, replacing $(p^\ast, \Theta^\ast)$ by $(p, \Theta)$ and keeping in mind that $(p, \Theta) \in B$, we obtain
\begin{equation} \label{est_p_1}
\begin{aligned}
& \|p(t) \|_{H^2}^2+\|p_t(t) \|_{H^2}^2+ \tau \|p_{tt}(t) \|_{H^1}^2 \\
\leq  &\, C_p(R_1, R_2) \Big( \|p_0 \|_{H^2}^2+\|p_1 \|_{H^2}^2+ \tau \|p_{2} \|_{H^1}^2\\
&+ \int_0^t \big(\|p(s) \|_{H^2}^2+\|p_t(s) \|_{H^2}^2+ \tau \|p_{tt}(s) \|_{H^1}^2\big) \ds \Big).
\end{aligned}
\end{equation}
Adding up \eqref{est_theta_1} and \eqref{est_p_1}, we get 
\begin{equation} \label{}
\begin{aligned}
& \|p(t) \|_{H^2}^2+\|p_t(t) \|_{H^2}^2+ \tau \|p_{tt}(t) \|_{H^1}^2+ \|\Theta(t) \|_{H^2}^2 +\int_0^t \|\Theta(s) \|_{H^3}^2 \ds\\
&+\|\Theta_t(t) \|_{H^1}^2+ \int_0^t \|\Theta_{t}(s) \|_{H^2}^2 \ds \\
\leq  &\, \tilde{C}(R_1, R_2) \Big( \|p_0 \|_{H^2}^2+\|p_1 \|_{H^2}^2+ \tau \|p_{2} \|_{H^1}^2+\|\Theta_0 \|_{H^3}^2\\
&+ \int_0^t \big(\|p(s) \|_{H^2}^2+\|p_t(s) \|_{H^2}^2+ \tau \|p_{tt}(s) \|_{H^1}^2\big) \ds \Big).
\end{aligned}
\end{equation}
Consequently, it suffices to apply Gronwall's inequality to the above bound to reach the estimate
\begin{equation} \label{}
\begin{aligned}
& \| p \|_{L^\infty H^2}^2+ \| p_t \|_{L^\infty H^2}^2 +\tau \| p_{tt} \|_{L^\infty H^1}^2+ \| \Theta \|_{L^2 H^3}^2\\
& +\| \Theta \|_{L^\infty H^2}^2+ \| \Theta_t \|_{L^\infty H^1}^2+\| \Theta_t \|_{L^2 H^2}^2\\
\leq &\, \tilde{C}(R_1, R_2)\exp(T)\big(\|p_0 \|_{H^2}^2+\|p_1 \|_{H^2}^2+ \tau \|p_{2} \|_{H^1}^2+\|\Theta_0 \|_{H^3}^2\big).
\end{aligned}
\end{equation}
This then concludes the proof of Theorem \ref{local_Existence_Thoerem}.

 %\bibliography{references_JMGT_Pen.bib}{}
%\bibliographystyle{siam} 
  
\end{document}